\documentclass[10p,a4paper]{article}
\RequirePackage{color}
\RequirePackage{url}
\RequirePackage{ifpdf}
\RequirePackage{ifthen}
\RequirePackage{amsthm,amsmath}
\usepackage{subfigure}
\usepackage{amssymb,amsthm}
\RequirePackage{multirow}
\RequirePackage{hangcaption}
\usepackage[authoryear]{natbib}
\usepackage{hypernat}
\usepackage{graphicx}
\numberwithin{equation}{section}
\newlength{\ei}\ei=0.0138888889em
\theoremstyle{plain}
\newtheorem{Thm}{Theorem}[section]   
\newtheorem{Prop}[Thm]{Proposition}
\newtheorem{Lem}[Thm]{Lemma}
\newtheorem{Cor}[Thm]{Corollary}

\theoremstyle{rem}
\newtheorem{Rem}[Thm]{Remark}

\newcommand{\rfia}[1]{\makebox[\parindent][l]{%
                     \makebox[0em][r]{\rm(}\sf#1\rm)}}
\newcounter{ABCcB}
\newcommand{\theABCcC}{\alph{ABCcB}}

\newenvironment{ABC}{\begin{list}{
  \rfia{\theABCcC}}{\usecounter{ABCcB} \topsep 0ex \partopsep 0ex \itemsep0ex
  \parsep=\parskip \leftmargin 0em \rightmargin 0em \itemindent=\parindent
  \listparindent=\parindent  \labelsep 0.5em \labelwidth 0.5em }}{\end{list}}

\newcommand{\href}[2]{{#2}}
\newcommand{\hreft}[1]{{\href{#1}{\tt\url{#1}}}}
\newenvironment{keyword}[1]{\textsc{Keywords: }\textsf{#1}}{}
\newcommand{\MSC}[1]{\texttt{#1}}
\newcommand{\Ew}{\mathop{\rm {{}E{}}}\nolimits} 
\newcommand{\Cov}{\mathop{\rm Cov}\nolimits}    
\newcommand{\ve}{\varepsilon}

\newcommand{\SSs}{\scriptscriptstyle}
\newcommand{\Ts}{\textstyle}

\newcommand{\tr}{\mathop{\rm{} tr{}}}

\newcommand{\R}{\mathbb R}

\newcommand{\N}{\mathbb N}
\newcommand{\EM} {{\mathbb I}}

\newcommand{\Tfrac}[2]{\textstyle\frac{#1}{#2}}

\newlength{\SyW}

\newcommand{\boma}[1]{\mbox{$#1$}}
\def\beq{\begin{equation}}
\def\eeq{\end{equation}}

\begin{document}
\title{Robust Kalman tracking and smoothing with propagating
 and non-propagating outliers}
\author{Peter Ruckdeschel\\
{\small Fraunhofer ITWM, Abt.\ Finanzmathematik,}\\
{\small Fraunhofer-Platz 1, 67663 Kaiserslautern, Germany}\\
{\small and TU Kaiserslautern, AG Statistik, FB.\ Mathematik,}\\
{\small P.O.Box 3049, 67653 Kaiserslautern, Germany}\\
{e-mail: }\href{mailto:peter.ruckdeschel@itwm.fraunhofer.de}{\tt\small peter.ruckdeschel@itwm.fraunhofer.de}
\\[2ex]
Bernhard Spangl\\
{\small University of Natural Resources and Applied Life Sciences,}\\
{\small Institute of Applied Statistics and Computing,}\\
{\small Gregor-Mendel-Str.\ 33, 1180 Wien, Austria}\\
{e-mail: }\href{mailto:bernhard.spangl@boku.ac.at}{\tt\small bernhard.spangl@boku.ac.at}
\\[2ex]
Daria Pupashenko\\
{\small Hochschule Furtwangen, Fak. Maschinenbau und Verfahrenstechnik,}\\
{\small Jakob-Kienzle-Straße 17, 78054 Villingen-Schwenningen, Germany}\\
{\small and TU Kaiserslautern, AG Finanzmathematik, FB.\ Mathematik,}\\
{\small P.O.Box 3049, 67653 Kaiserslautern, Germany}\\
{e-mail: }\href{mailto:pud@hs-furtwangen.de}{\tt\small pud@hs-furtwangen.de }
\\[2ex]
}
\date{\today}
\maketitle
\begin{abstract}
A common situation in filtering where classical Kalman filtering
does not perform particularly well is tracking in the presence
of propagating outliers. This calls for robustness understood in
a distributional sense, i.e.; we enlarge the distribution
assumptions made in the ideal model by suitable neighborhoods.
Based on optimality results for distributional-robust Kalman
filtering from \citet{Ru:01D,Ru:10a}, we propose new robust
recursive filters and smoothers designed for this purpose
as well as specialized versions for non-propagating outliers.
We apply these procedures in the context of a GPS problem
arising in the car industry. To better understand these filters,
we study their behavior at stylized outlier patterns
(for which they are not designed) and compare them to other
approaches for the tracking problem. Finally, in a simulation
study we discuss efficiency of our procedures in comparison
to competitors.
\end{abstract}
%
%
\section{Introduction}
\paragraph{Motivation}
State space models (SSMs) build a flexible but still manageable class of dynamic
models, and together with corresponding attractive procedures
as the Kalman filter and its extensions they provide an immensely useful
tool for a wide range of applications.
In this paper we are focussing on an application in engineering
in the context of a GPS problem arising in the car industry with
different linear and non-linear state space formulations.

It is common knowledge for long that most of the classical procedures,
 used in this domain suffer from lack of robustness, in particular the
 Kalman filters and smoothers which we concentrate on.

The mere notion of robustness however in filtering context is not canonic,
with the general idea to describe stability of the procedure w.r.t.\
variations of the ``input parameters''. The choice of
the ``input parameters'' to look at varies from notion to notion
passing from initial values to bounds on user-defined controls to
distributional assumptions. As general in robust statistics, we are concerned with
distributional or outlier robustness; i.e.; our input parameters
are the model distributions and defining suitable neighborhoods
about the ideal model we allow for deviations in the respective
assumptions which capture various types of outliers.

 The amount of literature on robustifications of these procedures is huge, so we do not
attempt to give a comprehensive account here. Instead, we refer to the surveys given
in \citet{Er:78}, \citet{Ka:Po:85}, \citet{Du:St:87},
\citet{Sch:M:94}, \citet{Ku:01}, and to some extent \citet[Sect.~1.5]{Ru:01D}.
In Section~\ref{rLSsec.AO} below, we list references related more
closely to our actual approach.

\paragraph{Problem statement}
In our context these outliers may be system-endogenous (i.e., propagating)
or -exogenous (i.e. non-propagating), which induces the somewhat conflicting
goals of tracking and attenuation.

If we head for robust optimality in the sense of minimizing mean squared
error on neighborhoods, the standard outlier models from \citet{F:72}
pose problems barely tractable in closed form, although some approximate solutions have
been given, e.g., in \citet{M:M:77}. A way out is given by outliers
of substitutive type where closed-form saddle-points
could be derived for the attenuation problem in filtering in \citet{Ru:01D}
which involve the hard-to-compute ideal conditional expectation. If however
this conditional expectation is linear, we come up with an easy and fast
robustification (rLS) which also is available in multivariate settings.
Although exact linearity can be disproved, approximate linearity usually
holds, at least in a central region of the distribution.

This approach has been generalized to simple tracking problems where
we can completely observe the states (i.e., the observation matrix
is invertible) in \citet{Ru:10a,Ru:10b}.
In this paper, we consider the general situation, i.e., tracking in
situations where observation dimension is lower than the one of the 
states.
To this end we propose a suitable generalization of the rLS to this
tracking problem and discuss its limitations.

To be able to apply these techniques in the EM algorithm for parameter estimation in
state space models (see \citet{Sh:St:82}) we also generalize our
procedures to the fixed interval smoothing problem where we come
up with a strictly recursive solution as in the classical case going
backward from the last state to first one.

As needed in the application, we also make available our techniques
for Extended Kalman filtering and smoothing, which to the best
of our knowledge is novel.

\paragraph{Organization of the paper}
Our paper is organized as follows:
In the text, the focus is on an understanding of the procedures and their
behavior at data, so we delegate mathematical results to the appendix.
We put some emphasis though on complete specification of the procedures.

In Section~\ref{GenSetup}, we introduce
the general framework of linear, time varying SSMs and their extension
as to outlier models. In Section~\ref{rLSsec} we define the filtering
and smoothing problems and recall the classical optimal solutions,
i.e., the Kalman filter and smoother and contrast these to our robustifications
based on the rLS. This section also introduces the central procedure of this paper, the rLS.IO,
and gives the necessary generalizations
to cover the smoothing and the Extended Kalman filter case, and,
for comparison, introduces other robust non-parametric alternatives.
In Section~\ref{Stylized}, we study their behavior in the ideal
situation and at stylized outlier patterns for which they are not necessarily designed.
In particular we demonstrate that the tracking problem at a model with a
non-trivial null space of the observation matrix leads to problems
where, at least at filtering time, for certain outliers, any procedure must fail.
Section~\ref{ApplSect} presents an application of our procedures
in the car industry dealing with GPS problems arising there. We sketch
three different SSMs used to model this situation, in particular including a
non-linear one, making robust Extended Kalman filters indispensable.
Section~\ref{SimSect} provides a comparative simulation
study in which the outliers are generated according to the ones
for which the procedures have been defined.
Conclusions are gathered in Section~\ref{SumSect}.

As to the mathematical
details to our procedures, Appendix~\ref{OptimClKF} presents some
derivation of the optimality of the Kalman filter among all
linear filters without any rank-conditions for the arising matrices
and under more general semi-norms than the Euclidean one.
Appendix~\ref{rLS.AO.opt} gives a brief account on the optimality
of the rLS, citing relevant results. Finally, Appendix~\ref{OptRLSIO} contains the
mathematical core of this paper, giving the necessary generalizations
in order to translate the optimality results from the AO case
to the tracking problem.

\section{Setup}\label{GenSetup} 
\subsection{Ideal model} \label{idmodel}
We consider a linear SSM consisting in an  unobservable $p$-dimensional
state $X_t$ evolving according to a possibly time-inhomo\-geneous
vector autoregressive model of order $1$  with
innovations $v_t$ and transition matrices $F_t \in\R^{p\times p}$, i.e.,
\begin{equation} \label{VAR1}
X_t=F_t  X_{t-1}+ v_t
\end{equation}
a $q$-dimensional observation $Y_t$ which is a linear transformation
$X_t$ involving an additional observation error $\ve_t$ and a corresponding
observation matrix $Z_t\in\R^{q\times p}$,
\begin{equation}
Y_t =Z_t  X_t+ \ve_t \label{linobs}
\end{equation}
In the ideal model we work in a Gaussian context, that is we assume
\begin{align}
&  v_t \stackrel{\rm \SSs indep.}{\sim} {\cal N}_p(0,Q_t), \qquad
\ve_t \stackrel{\rm \SSs indep.}{\sim} {\cal N}_q(0,V_t),  \qquad
 X_0  \stackrel{\phantom{\rm \SSs indep.}}{\sim}  {\cal N}_p(a_0,Q_0),
 \label{normStart}\\
&{\mbox{$ \{X_0, v_s, \ve_t,  \;s,t\in\N\}$ stochastically independent}}
\end{align}
For this paper, we assume the hyper--parameters $F_t,Z_t,Q_t,V_t,a_0$ to be known.
%

\subsection{Deviations from the ideal model} \label{devidmod}
%
As announced, these ideal model assumptions for robustness considerations
are extended by allowing (small) deviations, most prominently generated
by outliers.
In our notation, suffix ``${\Ts \rm id}$'' indicates the {\it id}eal setting,
``${\Ts \rm di}$'' the {\it di}storting (contaminating) situation, ``${\Ts \rm re}$''
the {\it re}alistic, contaminated situation.

\paragraph{AOs and IOs}
In time series context the most important distinction to be made
as to outliers is whether they propagate or not. To label these
different types, for historical reasons,
we use the terminology of \citet{F:72} (albeit in a somewhat more
general sense). Fox distinguishes {\em innovation outliers} (or IOs)
which enter the state layer and hence propagate
and {\em additive outliers} (or AOs) which only affect single observations
and do not propagate. Originally, AOs and IOs denote gross errors affecting
the observation errors and the innovations, respectively. We use these terms
in a wider sense: \textit{IO}s  also may cover level shifts
or linear trends which would not be included in the original definition.
Similarly \textit{AO} here will denote general exogenous outliers
which do not propagate.

More specifically, our procedures rLS.AO and rLS.IO defined below
both assume substitutive outliers, but should also provide protection
to some extent to arbitrary outliers of (wide-sense) AO respectively
IO type. To be precise, rLS.AO assumes a substitutive outlier (SO) model
already used by \citet{B:S:93} and \citet{Bi:Pa:94}, i.e.,
\begin{equation}
 Y^{\rm\SSs re} = (1-U)  Y^{\rm\SSs id} + U Y^{\rm\SSs di},
 \qquad U\sim {\rm Bin}(1,r) \label{YSO}
\end{equation}
for SO-contamination radius  $0\leq r\leq 1$
specifying the size of the corresponding neighborhood, and where
$U$ is assumed independent of
$(X,Y^{\rm\SSs id})$ and $(X,Y^{\rm\SSs di})$ as well as
 \begin{equation}\label{indep2}
 Y^{\rm\SSs di},\; X\quad \mbox{independent}
 \end{equation}
As usual, the contaminating distribution ${\cal L}(Y^{\rm\SSs di})$
is arbitrary, unknown and uncontrollable.

Similarly rLS.IO assumes that \eqref{VAR1} is split up into two steps,
\begin{equation}\label{IOSO}
\tilde X_t = F_t X^{\rm\SSs re}_{t-1} + v^{\rm\SSs id}_t, \qquad
X^{\rm\SSs re}_t = (1-\tilde U_t)  \tilde X_t + \tilde U_t X_t^{\rm\SSs di},
\qquad Y_t^{\rm\SSs re} = Z_t X^{\rm\SSs re}_t + (1-\tilde U_t)\ve^{\rm\SSs id}_t
\end{equation}
where $X_{t-1}^{\rm\SSs re}$ is the state according to
the contaminated past, $v^{\rm\SSs id}_t$ and $\ve^{\rm\SSs id}_t$ are
an ideally distributed innovation respectively observation error,
and $\tilde U_t$ and $ X^{\rm\SSs di}_t$ are defined
in analogy to $U_t$ and $ Y^{\rm\SSs di}$ (i.e., with independence
from all ideal distributions and the past).

\paragraph{Different and competing goals induced by AOs and IOs}

Due to their different nature, as a rule, a different reaction in
the presence of IOs and AOs is required. As AOs are exogenous, we would
like to ignore them as far as possible, damping their effect,
while when there are IOs, something has happened in the system, so
 the usual goal will be to  detect these  structural changes as fast as
 possible.

 A situation where both AOs and IOs may occur is more difficult,
 as we cannot distinguish IO from AO type immediately after a
 suspicious observation; it will be treated elsewhere.

\paragraph{Other deviation patterns}
Of course the two substitutive outlier types of Section~\ref{devidmod}
are by no means exhaustive: Trends and level shifts are not directly
covered, neither are patterns where the outlier mechanism may know
about the past like the patchy outliers described in \citet{Ma:Yo:86}.
Another type of outliers are outliers in the oscillation behavior,
in both amplitude/scale and frequency, or more generally spectral
outliers as arising in robustly estimating
spectral density functions, compare \citet{Fr:Po:83,Fra:85}, and
\citet{Span2008}.
We try to account  at least for some of them in Section~\ref{Stylized}.

\section{Kalman filter and smoother and robust alternatives} \label{rLSsec}
\paragraph{Filter problem}
The most important problem in SSM formulation is the reconstruction of
the unobservable states $X_t$ by means of the observations $Y_t$.
For abbreviation let us denote
\begin{equation}
Y_{1:t}=(Y_1,\ldots,Y_t), \quad Y_{1:0}:=\emptyset
\end{equation}
Using MSE risk, the optimal reconstruction is the solution to
\begin{equation}
\Ew \big| X_t-f_t\big|^2 = \min\nolimits_{f_t},\qquad
  f_t \mbox{ measurable w.r.t.\ } \sigma(Y_{1:s})  \label{MSEpb}
\end{equation}
We focus on filtering ($s=t$) in this paper, while $s<t$ makes for a
prediction, and $s>t$ for a smoothing problem.
%
\subsection{Classical method: Kalman filter and smoother} \label{classKalmss}
The general solution to \eqref{MSEpb}, the corresponding
conditional expectation  $\Ew[ X_t|Y_{1:s}]$ usually is rather
expensive to compute. Hence as in  the Gauss-Markov setting,
restriction to linear filters is a common way out.
In this context, \citet{Kal:60} introduced a recursive scheme to
compute this optimal linear filter reproduced here for
later reference:
\begin{align}
\!\!\!\!\!\!\mbox{Initialization: }&\!\!\!\!&
 X_{0|0} &= a_0,\qquad &\Sigma_{0|0}&=Q_0 \label{bet1}
 \\
\!\!\!\!\!\!\mbox{Prediction: }&\!\!\!\!
& X_{t|t-1}&= F_t  X_{t-1|t-1} , \qquad &\Sigma_{t|t-1}
&=F_t\Sigma_{t-1|t-1}F_t^\tau + Q_t\label{bet2}
\\
\!\!\!\!\!\!\mbox{Correction: }
&\!\!\!\!& X_{t|t}&=  X_{t|t-1} + K_t \Delta Y_t, \label{bet3}
\!\!\!\quad&  \Delta Y_t &= Y_t-Z_t x_{t|t-1},\nonumber\\
&&K_t&=\Sigma_{t|t-1}Z_t^\tau C_t^{-1},
\nonumber
\!\!\!\quad&  \Sigma_{t|t} &= (\EM_p-K_tZ_t)\Sigma_{t|t-1},\nonumber\\
&&C_t&=Z_t \Sigma_{t|t-1}Z_t^\tau + V_t
\end{align}
where $\Sigma_{t|t}=\Cov(X_t-X_{t|t})$, $\Sigma_{t|t-1}=\Cov(X_t-X_{t|t-1})$,
and $K_t$ is the so-called \textit{Kalman gain}.

The Kalman filter has a clear-cut structure with an initialization, a
prediction, and a correction step. Evaluation and interpretation is
easy, as all steps are linear.
The strict recursivity / Markovian structure of the state equation
allows one to concentrate all information from the past useful for
the future in  $X_{t|t-1}$.\\
This linearity is also the reason for its non-robustness, as
observations $y$ enter unbounded into the correction step. A good
robustification has to be bounded in the observations, otherwise
preserving the advantages of the Kalman filer as
far as possible.
\subsection{A robustification of the least squares solution (rLS)} \label{rLSsec.AO}
The idea of the procedures we discuss in this paper are based on
{{\it r}obustifying {\it r}ecursive  {\it L}east {\it S}quares: rLS},
a filter originally introduced for AOs only, compare \citet{Ru:00,Ru:01D}.
Starting from a different route of translating regression M estimators
to SSM context, \citet{B:D:83,B:D:87,C:R:91} arrive at similar weight functions, albeit
with an iterative procedure to solve for the M equations. More recently,
restricted to the particular SSM class of Holt-Winter-Forecasting,
\citet*{G:F:C:10} take up this idea and use regression M estimators, which,
in addition to the present rLS approach involve recursive scale estimation.
The closest recent approach stems from \citet{C:H:11}, who---without translating
their clipping to a bias side condition for an explicit outlier model---show
a Lemma-5-type optimality (see Problem~\eqref{Lem5SO} below) for their
procedure (with application to the special SSM case of exponential 
smoothing). Their solution is different from ours only in the choice of the norm
used for clipping: They use a diagonal weighting matrix $W$ whereas, in the AO case, we use
no weighting; a weighting does appear though in the general IO solution,
compare Lemma~\ref{lem1} below.

In this paper, we also consider an IO-robust version of this concept, and
hence here, in addition to the original definition, we append the suffixes
``.AO'' and ``.IO'' to distinguish the two versions.
Let us begin with (wide-sense) AOs.

With only AOs, there is no need for robustification in the initialization
and prediction step, as no (new) observations enter.
As introduced in \citet{Ru:00}, we robustify the correction step,
replacing $K\Delta Y$ by a Huberization $H_b(K\Delta Y)$ where
 $H_b(x) = x \min\{1, b/\big|x\big|\}$ for some suitably chosen clipping
height $b$ and some suitably chosen norm, natural candidates being Euclidean
and Mahalanobis norm. It turns out that only little is gained if we account
for this modification in the recursion for the filter covariance, so we leave
this unchanged. That is, the only modification in the correction step becomes
\begin{equation} \label{HbDef}
X_{t|t}=  X_{t|t-1} + H_b(K_t \Delta Y_t)
\end{equation}
While this is a bounded substitute for the correction step in the classical Kalman filter, it still
remains reasonably simple, is non iterative and hence especially useful
for online-purposes.

However it should be noted that, departing from the Kalman filter and
at the same time insisting on strict recursivity, we possibly exclude
``better'' non-re\-cur\-sive procedures. These procedures on the other
hand would be much more expensive to compute.

As sketched in Appendix~\ref{rLS.AO.opt}, it can be shown that rLS not
only is plausible, but also has some optimality properties. These
are not the focus of this paper, though, and proofs of these
properties appear elsewhere.

\paragraph{Choice of the clipping height $b$}
For the choice of  $b$, we have two proposals. Both are based on the
simplifying assumption that $\Ew_{\rm\SSs id}  [\Delta X | \Delta Y]$
is linear, which in fact turns out to  only be approximately correct.
The first one chooses $b=b(\delta)$ according to an \citet{Ans:60}
criterion,
    \begin{equation}
     \Ew_{\rm\SSs id} \big|\Delta  X-H_b(K \Delta Y)\big|^2
         \stackrel{!}{=}(1+\delta)
          \Ew_{\rm\SSs id} \big|\Delta  X-K \Delta Y\big|^2
         \label{deltakrit}
    \end{equation}
where $\delta$ may be interpreted as ``insurance premium''
to be paid in terms of efficiency.

The second criterion uses the radius $r\in[0,1]$  of the
neighborhood ${\cal U}^{\rm\SSs SO}(r)$ (defined in \eqref{U-SO}) and
determines $b=b(r)$ such that
    \begin{equation}
     (1-r) \Ew_{\rm\SSs id} (|K \Delta Y |-b)_+ \stackrel{!}{=} r b
         \label{deltakrit2}
    \end{equation}
This produces the minimax-MSE procedure for ${\cal U}^{\rm\SSs SO}(r)$
(compare Section~\ref{rLS.AO.opt}). Generalizing ideas of
\citet*{R:K:R:08}, this criterion can be extended to
situations where we only know that the radius lies in some interval,
but we do not work this out here; for details see \citet{Ru:10a}.

\begin{Rem}
It turns out that in filtering context, the second criterion reflects
much better the problem inherent difficulty of a robustification:
While $10\%$ efficiency loss in the ideal model can be unreachable,
because totally ignoring the new observation $\Delta Y$
would only ``cost'' $5\%$, it may be much too little to produce a sizeable
effect in other models. A radius of $0.1$ however seems to lead to reasonable
choices of $b$ in most models.
\end{Rem}

\subsection{rLS.IO} \label{rLSsecIO}
As noted, in the presence of IOs, we want to follow an IO outlier as
fast as possible.

Optimality results on distributional neighborhoods for the IO-case with
the goal of a faster tracking to the best of our knowledge have not been
considered by other authors so far.

The  Kalman filter in this situation does not behave as
bad as in the AO situation, but still tends to be too inert. To improve upon
this, let us first simplify our model to the situation where
we have an unobservable but interesting state $X\sim P^X(dx)$ and where
instead of $X$ we rather observe the sum
\begin{equation} \label{simpAdd}
Y=X+\ve
\end{equation}
This equation reveals a useful symmetry
of $X$ and $\ve$: Apparently
\begin{equation} \label{simpEx}
\Ew[X|Y] = Y-\Ew[\ve|Y]
\end{equation}
Hence we follow $Y$ more closely if we damp estimation of $\ve$, for
which we use the rLS-filer. We should note that doing so, we rely
on ``clean'', i.e., ideally distributed errors $\ve$. With the
obvious replacements, our optimality results from the appendix
translate word by word to a corresponding results for IOs,
compare \citet[Thm.~4.1]{Ru:10b}.

In analogy to the definition of the rLS in equation~\eqref{HbDef},
we set up an IO-robust version of the rLS as follows: We retain
the initialization and prediction step of the classical Kalman filter
and for the correction step, as proved in Corollary~\ref{CorIO}, we note
that in the ideal model, as
$\Ew[\ve_t|\Delta Y_t] = (\EM_q-Z_t K_t)\Delta Y_t$,
the correction step can also be written as
\begin{equation}
X_{t|t} = X_{t|t-1} + Z_t^{\Sigma} \big(\Delta Y_t - \Ew[\ve_t|\Delta Y_t] \big),
\end{equation}
where $Z_t^\Sigma$ is a suitably generalized inverse for $Z_t$,
minimizing the respective MSE among all $\Ew[\ve|\Delta Y]$-measurable
(and square integrable) functions, i.e.,
\begin{equation} \label{ZSigma}
  Z_t^\Sigma:=\Sigma_{t|t-1} Z^\tau_t (Z^\tau_t \Sigma_{t|t-1}Z_t)^-
\end{equation}
The IO robustification then simply consists in replacing
$\Ew[\ve_t|\Delta Y_t]$ by  $H_b((\EM_q-Z_t K_t)\Delta Y_t)$,
i.e., the rLS.IO correction step as
\begin{equation}
X_{t|t} = X_{t|t-1} + Z_t^{\Sigma} [\Delta Y_t - H_b((\EM_q-Z_t K_t)\Delta Y_t)]
\end{equation}
Here the same arguments for the choice of the norm and the clipping
height apply as for the AO-robust version of the rLS.

To better distinguish IO- and AO-robust filters, let us call the
IO-robust version \emph{rLS.IO} and (for distinction) the AO-robust
filter \emph{rLS.AO} in the sequel.

\begin{Rem}
It is worth noting that also our IO-robust version is a filter,
hence does not use information of observations made after the state
to reconstruct; rLS.IO is strictly recursive and non iterative,
hence well-suited for online applications.
\end{Rem}
\subsection{Extended Kalman filter} \label{EKF}
In the application we are heading for, in Model~(M3) below, the SSM given by
equations~\eqref{VAR1} and
\eqref{linobs} is only a linearization of a nonlinear SSM given by
\begin{align}
X_t&=f_t(X_{t-1},u_t,v_t)\\
Y_t&=z_t(X_t,w_t,\ve_t)
\end{align}
for smooth, known functions $f_t$ and $z_t$ in the states $X_t$, in the innovations
$v_t$, in the observation errors $\ve_t$ and some user defined controls $u_t$ and
$w_t$. With $\bar v_t=\Ew v_t$, $\bar \ve_t=\Ew \ve_t$, this is linearized to give
the following \textit{Extended Kalman filter}  taken from \citet{W:vM:02}
\begin{align}
\!\!\!\!\!\!\mbox{Initialization: }&\!\!\!\!&
 X_{0|0} &= a_0,\!\!\!\qquad\!\!\! &\Sigma_{0|0}&=Q_0 \label{bet1EKF}
 \\
\!\!\!\!\!\!\mbox{Prediction: }&\!\!\!\!
& X_{t|t-1}&= f_t( X_{t-1|t-1},u_t,\bar v_t), \!\!\!\qquad\!\!\! &\Sigma_{t|t-1}
&=F_t\Sigma_{t-1|t-1}F_t^\tau + B_tQ_tB_t^\tau\label{bet2EKF}
\\
\!\!\!\!\!\!\!\!\!\!\!\!& \mbox{for}\!\!\!\!\!\!&F_t&=\Tfrac{\partial}{ \partial x}f_t(x,u_t,\bar v_t)|_{_{X_{t-1|t-1}}},\!\!\!\qquad\!\!\!
&B_t&= \Tfrac{\partial}{ \partial v}f_t(X_{t-1|t-1},u_t,v)|_{_{\bar v_{t}}}\nonumber\\
\!\!\!\!\!\!\mbox{Correction: }
&\!\!\!\!& X_{t|t}&=  X_{t|t-1} + K_t \Delta Y_t, \label{bet3EKF}
\!\!\!\quad&  \Delta Y_t &= Y_t-z_t (x_{t|t-1},w_t,\bar \ve_t),\nonumber\\
&&K_t&=\Sigma_{t|t-1}Z_t^\tau C_t^{-1},
\nonumber
\!\!\!\quad&  \Sigma_{t|t} &= (\EM_p-K_tZ_t)\Sigma_{t|t-1},\nonumber\\
&&C_t&=Z_t \Sigma_{t|t-1}Z_t^\tau + D_t V_t D_t^\tau,\\
\!\!\!\!\!\!\!\!\!\!\!\!& \mbox{for}\!\!\!\!\!\!&Z_t&=\Tfrac{\partial}{ \partial x}z_t(x,u_t,\bar v_t)|_{_{X_{t|t-1}}},\!\!\!\qquad\!\!\!
&D_t&= \Tfrac{\partial}{ \partial \ve}z_t(X_{t|t-1},u_t,\ve)|_{_{\bar \ve_{t}}}\nonumber
\end{align}
As to the robustification of this algorithm by rLS.AO and rLS.IO, as in the linear case,
we simply replace the term $K_t \Delta Y_t$ by $H_b(K_t \Delta Y_t)$ in the AO case
and by $Z_t^\Sigma(\Delta Y_t- H_b((\EM_q-Z_tK_t)\Delta Y_t)$ in the IO case.

\subsection{A robust smoother} \label{ssmooth}
In many situations, in particular for an application of the EM-algorithm
to estimate the hyper-parameters, it is common use to enhance the filtered
values $X_{t|t}$ in retrospective, accounting for the information of
$Y_{t+1:T}$ which is available in the mean time. To retain
recursivity, we use a corresponding backward recursion as to be found in
\citet[Sec.7.4, (4.5)]{A:M:90}:
\begin{equation} \label{smeq}
X_{t|T} = X_{t|t}+ J_{t} (X_{t+1|T}-X_{t+1|t}), \qquad J_{t} =
\Sigma_{t|t} F_{t}^\tau \Sigma_{t+1|t}^{-1}
\end{equation}
with respective update for the smoothing covariance,
\begin{equation}
\Sigma_{t|T} = \Sigma_{t|t}+J_t(\Sigma_{t+1|T}-\Sigma_{t+1|t})J_t^\tau
\end{equation}
Writing \eqref{smeq} as
$X_{t|T} - X_{t|t}=J_t [(X_{t+1|T}-X_{t+1|t+1})+(X_{t+1|t+1}-X_{t+1|t})]$,
we see that the first summand in the brackets of the RHS is just the preceding
iterate of the LHS in the recursion and the second summand is just the
already robustified increment of the correction step in the filter. Hence,
sticking to our outlier models for IO and AO contamination which independently
affect single $X_t$'s and $Y_t$'s the modification only has to be done
in the (new) treatment of $\Delta Y_t$, i.e., in the second summand,
so there is no further need for robustification in
the backwards loop.
\subsection{A nonparametric approach: median based filters} \label{hybridFsf}
A nonparametric approach to assess the robustness issues met with Kalman filtering,
which does not use any state space formulation is to use median-type filters for this purpose.
Standard median filters remove outliers and preserve level shifts but their
output does not properly represent linear trends.
The repeated median (RM) has been developed for the extraction of monotonic trends
with intercept and slope  from a time series.

In order to combine advantages of several specifications of location and regression
based median-type filters, a variety of hybrid filters has been introduced in
\citet*{F:B:G:06}. From these filters, in Section~\ref{SimSect}, we use
the regression based and predictive approach {\tt PRMH}. In addition, we use
a location based approach {\tt MMH} for smoothing (notation {\tt MMH} taken
over from {\sf R} package {\tt robfilter}, \citet{F:S:08}).

Because of their high breakdown points for window width $k$ of
$(\lfloor k/2\rfloor +1)/n$ for {\tt PRMH} and $(2\lfloor k/2\rfloor +1)/n$
for {\tt MMH}, see \citet*{F:B:G:06}, these hybrid RM-type filters and smoothers
 are prominent candidates for a first sweep over the data.

For an automatic choice of the window width,
the {\it A}daptive {\it O}n-line {\it R}epeated {\it M}edian {\it R}egression (ADORE) filter
\citep{Schett:09} provides robust on-line extraction by a moving window technique with
adaptive window width selection.
%

In our setting, the median-based filters as they stand can only be applied
to situations with one-dimensional observations; generalizations to
higher dimensions beyond mere coordinate-wise treatment have been introduced
in \citet{Schett:09}, but are not dealt with in this paper.
Being non-parametric, these filters know nothing about an underlying
SSM, so cannot be used for state reconstruction, except for the case
that $Z_t=1$ (or for $p=q$, $\EM_q$ in higher dimensions).
So we include these filters only in the steady state examples  of
Sections~\ref{ChInOsc}, \ref{Sec:NonObsAsp} and \ref{SimSect} (Model~(SimB)).

\section{Behavior of the filters at stylized outlier situations} \label{Stylized}
In this section, we study the behavior of our filters in the ideal situation
and at stylized outlier patterns for which they are not necessarily
designed---these cover AOs and IOs  in the original sense, their
behavior at trends and level shifts and at changes in the oscillation behavior.
We also look at the effects of a non-trivial null space of the observation
matrix such that the information given by the observation does not suffice
to reconstruct the state.

\subsection{The ideal situation, spiky outliers and IOs} \label{idSit}

To start with, we study the behavior of the robustified (extended) Kalman filter
(EKF) in three different situations, namely no contamination, contamination by
(classical) IOs and AOs, and use the following hyper parameters for the SSM:
\begin{equation}
\begin{array}{llll}
\boma{a_0} = \left( \begin{array}{c} 20 \\ 0 \end{array} \right) ~, &
\boma{Q_0} = \left( \begin{array}{cc} 0 & 0 \\ 0 & 0
                         \end{array} \right) ~, \\
\boma{F_t} = \left( \begin{array}{cc} 1 & 1 \\ 0 & 0
                         \end{array} \right) ~, &
\boma{Z_t} = \left( \begin{array}{cc} 0.3 & 1 \\ -0.3 & 1
                         \end{array} \right) ~, &
\boma{Q_t} = \left( \begin{array}{cc} 0 & 0 \\ 0 & 9
                         \end{array} \right) ~, &
\boma{V_t} = \left( \begin{array}{cc} 9 & 0 \\ 0 & 9
                         \end{array} \right) ~.
\end{array}
\end{equation}
We note that the first coordinate of the above state process is a random
walk and therefore non-stationary, whereas the second coordinate is just
white noise.

The innovations $\boma{v}_t$ and the errors $\boma{\varepsilon}_t$
of the observation process are simulated
from a contaminated bivariate normal distribution
\begin{equation} \label{multivariate cnd}
\mathcal{CN}_2 (r, \boma{0}, \boma{R}, \boma{\mu}_c, \boma{R}_c) =
(1-r) \mathcal{N}_2 (\boma{0}, \boma{R}) +
    r \mathcal{N}_2 (\boma{\mu}_c, \boma{R}_c) ~,
\end{equation}
with $r$, the amount of contamination, set to 10\% and
where 
$\boma{R} = \boma{Q}_t$ in case of the innovations $\boma{v}_t$ and
$\boma{R} = \boma{V}_t$ in case of the observation errors
${\boma{\varepsilon}_t}$, and
the moments of the contaminating distribution are given by
\beq
\boma{\mu}_c^\tau=(25,30),\quad
\boma{R}_c= {\rm diag} (0.9,0.9)~.
\eeq

Then the bivariate filter estimates of the three simulated state-space models
are computed using the classical Kalman and rLS filters.
\begin{figure}[tp]
\centering
\includegraphics[width=\textwidth]{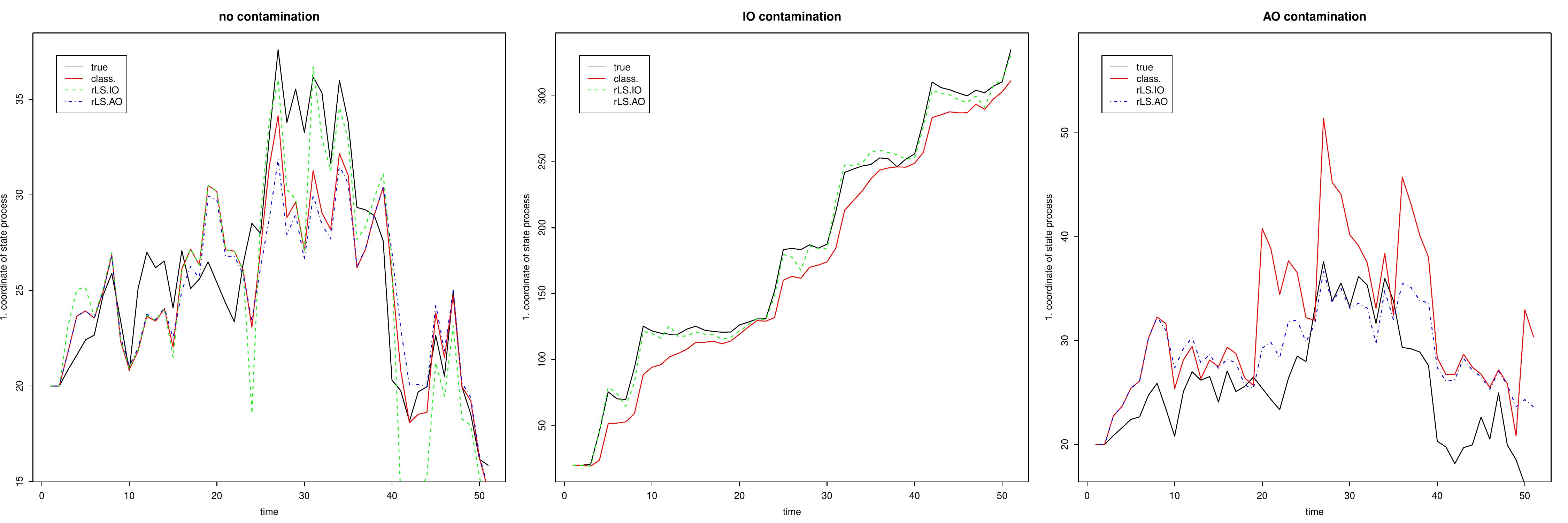}
\caption{Typical results of the state-space model for different
situations: left no contamination, middle IO contamination, right
AO contamination}
\label{exampleLinearEKF}
\end{figure}
In Figure~\ref{exampleLinearEKF}, we show typical realizations  of the 
state process together with their filter estimates 
for different contamination situations.
Only the first coordinate of the estimated state vector is plotted. 
The thick black line is the true state process, while
classical Kalman, rLS.IO, and rLS.AO filters are plotted as light red line,
dotted green line, and dot-dashed blue line, respectively.
We note that 
contamination by additive outliers cannot be seen directly
in Figure~\ref{exampleLinearEKF} because only the observation equation is
affected. However, the spikes, especially visible for the filter estimate 
of the classical Kalman filter, do indicate additive outliers.

\paragraph*{Ideal situation:}

The classical Kalman filter as well
as its robustified versions, yield almost the same results.

\paragraph*{IO contamination:}

The rLS.IO filter is able to follow the true state almost immediately, 
whereas the classical Kalman filter is only able to track the true 
state with a certain delay.

\paragraph*{Spiky outliers:}

The rLS.AO filter is not affected by additive outliers whereas 
the classical Kalman filter is prone to them.

\subsection{Changes in oscillation pattern and level shifts}\label{ChInOsc}

Next, in situations with outliers scaling on different frequencies, 
we compare our robustified versions of the Kalman filter to a 
non-parametric filtering method, i.e., the ADORE filter proposed 
by \citet{Schett:09}---with an automatic selection of the window 
width. 
As outlier situations, we consider (classical) IOs and AOs (with
contamination radius $r=10\%$ each) 
as well as the special endogenous case where  we substitute part of the state by
a completely artificial signal. 
For the SSM we use an AR(2) process, i.e., an autoregressive
process of order 2, as state process, and 
 the following hyper parameters:
\begin{equation}
\begin{array}{llll}
\boma{a_0} = \left( \begin{array}{c} 0 \\ 0 \end{array} \right) ~, &
\boma{Q_0} = \left( \begin{array}{cc} 0 & 0 \\ 0 & 0
                         \end{array} \right) ~, \\
\boma{F_t} = \left( \begin{array}{cc} 1 & -0.9 \\ 1 & 0
                         \end{array} \right) ~, &
\boma{Z_t} = \left( \begin{array}{cc} 1 & 0
                         \end{array} \right) ~, &
\boma{Q_t} = \left( \begin{array}{cc} 1 & 0 \\ 0 & 0
                         \end{array} \right) ~, &
\boma{V_t} = \left( \begin{array}{c} 1
                         \end{array} \right) ~.
\end{array}
\end{equation}
The innovations $\boma{v}_t$ in the IO situation
again stem from a contaminated bivariate normal distribution
given in (\ref{multivariate cnd}) with moments of the contaminating 
distribution given by
\beq
\boma{\mu}_c^\tau=(30,0),\quad
\boma{R}_c= \left( \begin{array}{cc} 0.1 & 0 \\ 0 & 0
                         \end{array} \right)  ~.
\eeq

The contaminating distribution of the errors  $\boma{\varepsilon}_t$ of
the observation process is $\mathcal{N} (10, 0.1)$.
Then, for all different situations the filter estimates
are computed using the rLS filter and the ADORE filter.
Moreover, also the classical Kalman filter estimates were calculated.

\begin{figure}[tp]
\centering
\includegraphics[width=\textwidth]{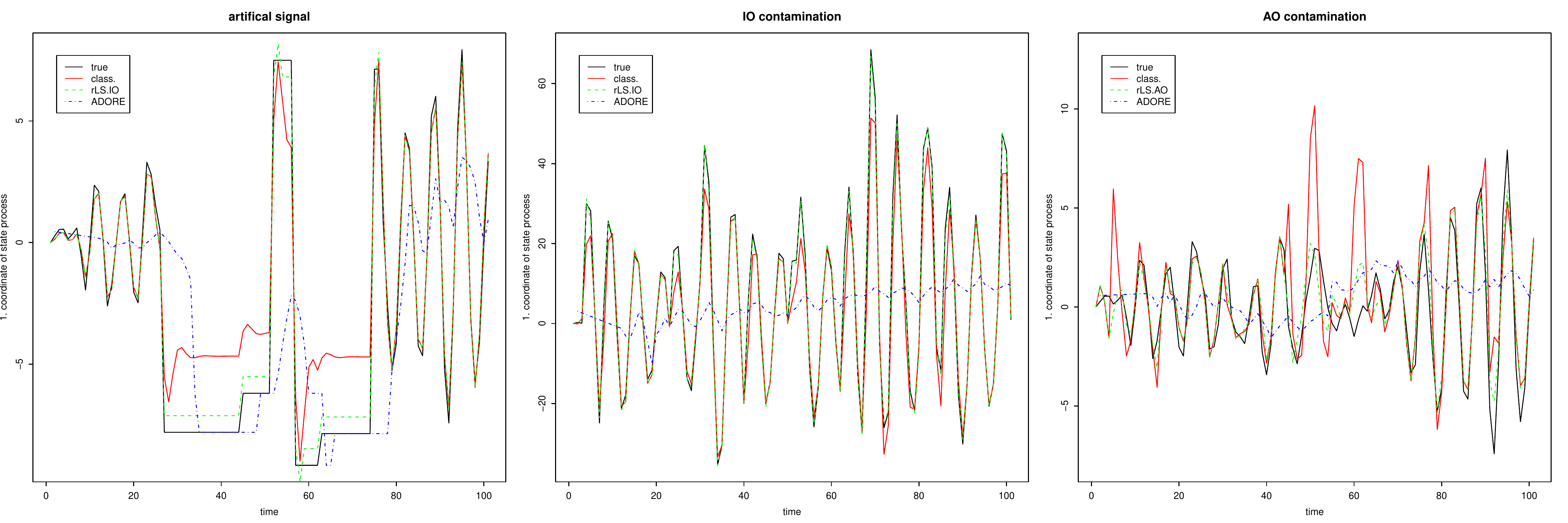}
\caption{Results of the simulated state-space model for different
situations: left contamination by artifical signal, middle IO
contamination, right AO contamination}
\label{exampleTrendEKF}
\end{figure}

In Figure~\ref{exampleTrendEKF}, the realizations of the state process, 
together with their filter estimates are displayed for different situations
of contamination. The black line is again the true state process, while 
classical Kalman, rLS, and ADORE filters are plotted by a red line, a  
dashed green line, and a dot-dashed blue line, respectively.

\paragraph*{Endogenous contamination by artificial signal:}

In this case, whole parts of the state
process are substituted by an artificial signal,
more specifically by a so called block signal \citep[cf., e.g.,][]{Do:Jo:94}
which is piecewise constant with random length and amplitude.
Here, the rLS.IO filter is able to track the state process best. 
The classical Kalman filter, as already seen in the previous section, 
is not able to follow the level shifts. Using the ADORE filter as 
non-parametric filtering method an obvious time delay in following 
the state signal can be seen.

\paragraph*{IO contamination:}

Here, the rLS.IO filter is able to follow the
true state, whereas the classical Kalman filter fails to track the spikes of
the state signal. The ADORE filter as non-parametric alternative only fits
an overall trend.
In general, non-parametric filters as ADORE are specialized to reveal trends,
trend changes or shifts of an underlying, possibly non-stationary signal in the
presence of outliers and, according to our experience, they extremely smooth the
underlying process.

\paragraph*{Spiky outliers:}

The rLS.AO filter is not affected by AOs whereas the classical Kalman filter 
is prone to them. The ADORE filter again shows the above mentioned 
properties and estimates more or less an overall trend of the 
underlying AR(2) process.

\subsection{Coping with non observed aspects}\label{Sec:NonObsAsp}
One important aspect of reconstructing states in SSMs is observability,
compare \citet[App.~C]{A:M:90}, as in general matrix $Z_t$ will have a 
non-trivial null space with the consequence that state signals falling
to this null space are not visible at filtering time. Smoothing may
to some extent relieve this problem with a certain time delay, when
subsequent transitions $F_s$ move the states in such a way
that they become visible to $Z_s$ at a later stage.

We illustrate this problem studying how the considered
versions of the Kalman filter cope with such non observed aspects
in the following setup: 
\begin{equation} \label{Eq:SimSect}
\begin{array}{llll}
T=50 ~, & F=\left(
  \begin{array}{ccc}
    1 & 1 & 0 \\
    0 & 1 & 1 \\
    0 & 0 & 0 \\
     \end{array}
\right) ~, & Z=\left(
  \begin{array}{ccc}
    1 & 0 & 0 \\
    0 & 0 & 1 \\
     \end{array}
\right) ~, &
Q = \textrm{diag}(0,0,0.001) ~,\\
& V = \textrm{diag}(0.1,0.001) ~, & a_{0}=(0, 0, 0)' ~, &
Q_0 = \textrm{diag}(1,0.1,0.001) ~.
\end{array}
\end{equation}
Here, we use as outlier specification the ones from Section~\ref{devidmod}
with $r_{\rm\SSs IO}=r_{\rm\SSs AO}=0.1$, i.e., \eqref{YSO} and \eqref{indep2} in
the AO case and \eqref{IOSO} in the IO case.
As contaminating distributions we chose
$X_t^{\rm\SSs di}\sim {\rm multiv.Cauchy}(0,Q)$
and $Y_t^{\rm\SSs di}\sim {\rm Cauchy}/1000$ (using {\sf R} packages
{\tt mvtnorm} \citep{mvtnorm} and {\tt MASS} \citep{V:R:02}).

\begin{figure}[tp]
\centering
\includegraphics[width=\textwidth]{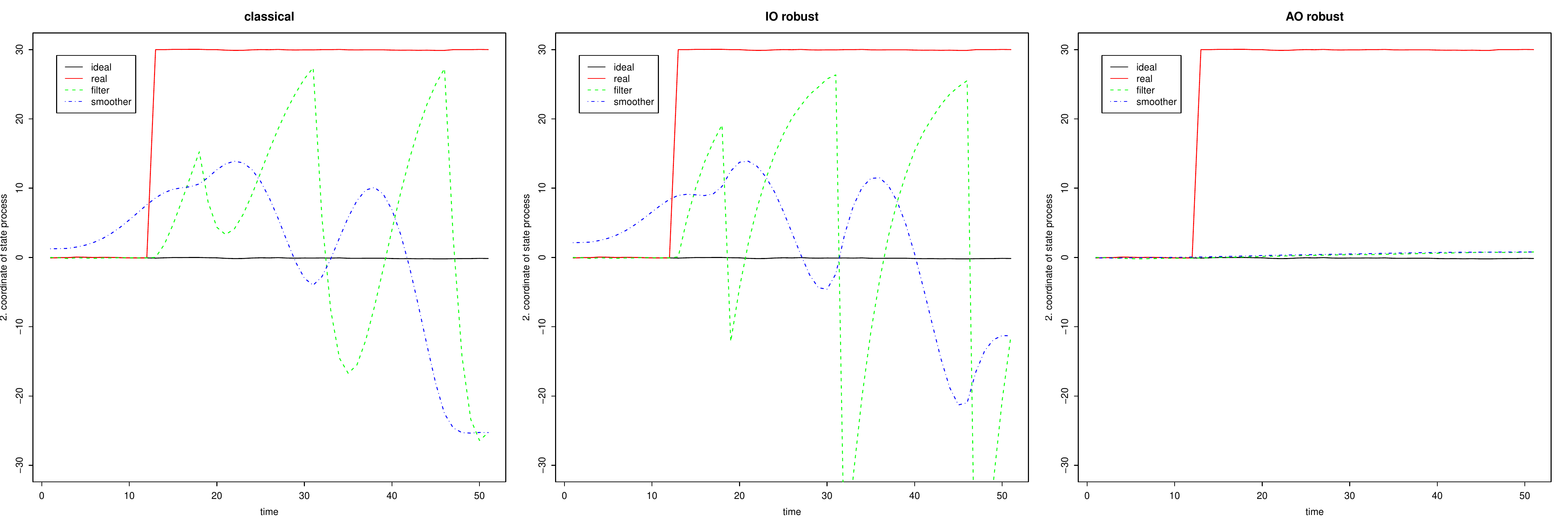}
\caption{Filter estimates of the simulated state-space model using  different
filter and smoothers: left classical Kalman filter and smoother, middle
IO-robust filter, right AO-robust filter}
\label{exampleRankEKF}
\end{figure}

In Figure~\ref{exampleRankEKF},  for different
versions of the rLS, we display one realization of the state process in the ideal and in the
real IO-contaminated situation
together with the filter and smoother estimates. 
The black line is
the true state process in the ideal situation, while the red line
represents the IO-contaminated state process, i.e., the real situation.
The dashed green line
represents the rLS filter, the dot-dashed blue line the corresponding
rLS smoother.

Only the second coordinate, which lies in  $\ker \boma{Z}_t$,
of the state process is plotted.

All three plots clearly reveal that none of the proposed filters is
able to cope with this situation, i.e., to correctly follow the level
shift in the second coordinate of the state process caused by an
IO-contamination.

\section{Application} \label{ApplSect}
In this section, we show an application of the procedures described above
to data from a cooperation with the department for Mathematical Methods in
Dynamics and Durability at Fraunhofer ITWM and Kaiserslautern University,
in particular with Nikolaus Ruf and J\"{u}rgen Franke.
This data is part of a larger project they are involved.
In the application, we deal with data taken from a vehicle moving on some track, which consists
of four data channels, i.e.,
\begin{itemize}
\setlength{\partopsep}{-5ex}
\setlength{\topsep}{-5ex}
\setlength{\itemsep}{-.3ex}
  \item time recorded in seconds ($[s]$);
  \item vehicle speed $\widetilde{sp_{t}}$, measured
  in meters per second ($[m/s]$);
  \item measured altitude $\widetilde{h}_{t}$ ($[m]$);
  \item pitch angle speed $\widetilde{\dot{\alpha}}_{t}$,
  measured in radians per second ($[rad/s]$).
\end{itemize}

We are interested in slope estimation, more precisely in the change of altitude
over distance, because this information cannot be captured so easily from
cartographic information.

The original data has been distorted by different factors, e.g.\
GPS measurement error, discreteness of the measurement process,
 change of the road surface, etc. If during the measurement the process signal
was lost for short time because of tunnels on the way or for some other
reasons, some part of the data is missing and we can observe e.g.\ sudden jumps
of the altitude or speed, what is impossible in practice
for the tracks. Such situations are very common in reality, therefore we are
interested in constructing reliable parametric models of such errors in order
to reconstruct the data.

With different simplifying assumptions
on the observed data---to some degree for illustration purpose---we construct
three models:
\begin{enumerate}
  \item[(M1)] {\bf A linear time-invariant model} with the following state vector and matrix:
\begin{equation*}
X_{t}=\left(
  \begin{array}{ccc}
    h_{t} \\
    \dot{h}_{t} \\
    c_{t} \\
     \end{array}
\right) \quad
F=\left(
  \begin{array}{ccc}
    1 & 1 & 0 \\
    0 & 1 & 1 \\
    0 & 0 & 0 \\
     \end{array}
\right) \quad
v_{t}\sim N_{3}(0,Q),
\quad
Q=\textrm{diag}(0,0,0.01),
\end{equation*}
where $c_{t}=sp_{t+1}\dot{\alpha}_{t},t=0,...,T-1$ denotes the compound increment
measured in ($[m/s]\times[rad/s]$). Note that by construction, $c_{T}$ is not defined,
therefore we additionally assume that $c_{T}=sp_{T}\dot{\alpha}_{T}$.

The hyper-parameters of the observation equation in this model are the following
\begin{equation*}
Y_{t}=\left(
  \begin{array}{cc}
    \widetilde{h}_{t} \\
    \widetilde{c}_{t} \\
     \end{array}
\right) \quad
Z=\left(
  \begin{array}{ccc}
    1 & 0 & 0 \\
    0 & 0 & 1 \\
     \end{array}
\right) \quad
\varepsilon_{t}\sim N_{2}(0,V),
\quad V=\textrm{diag}(5,0.01).
\end{equation*}
In reality, the variances have been obtained by an application of
 an EM-type algorithm as the one by \citet{Sh:St:82}.
 We mention that
this EM-Algorithm is all but robust, being based on classical estimators
for first and second  moments. This robustness issue will be discussed  elsewhere
in more detail.

 In our application, though, we inspected the norms of the observation residuals $\Delta Y_t$
and the estimated  innovations $\hat v_t$ and errors $\hat \varepsilon_t$ which were obtained from the
 rLS-filters and -smoothers applied to the real data.
 Among these obtained $\|\Delta Y_t\|$, $\|\hat v_t\|$, $\|\hat \varepsilon_t\|$, we did not observe
 outstanding values, which indicates that application of the classical
 EM-Algorithm in this case should be justified.

The initial distribution of the state vector here is defined as follows
\begin{equation*}
a_{0}=(h_{1}, 0, 0)', \quad Q_{0}=\textrm{diag}(5,1,0.01).
\end{equation*}

In Figure~\ref{plot31}, we plot the actual observations together with the
respective reconstruction by means of the classical Kalman filter, the rLS.AO,
and the rLS.IO in Models~(M1) and (M2) defined below, with
a zoom-out of observations 100--130 to better distinguish the different models.
The classical reconstruction is not clearly visible at the plot, since
in case of these data, it almost coincides with the rLS.AO.

\begin{figure}[h!]
\center{\includegraphics[width=0.8\linewidth]{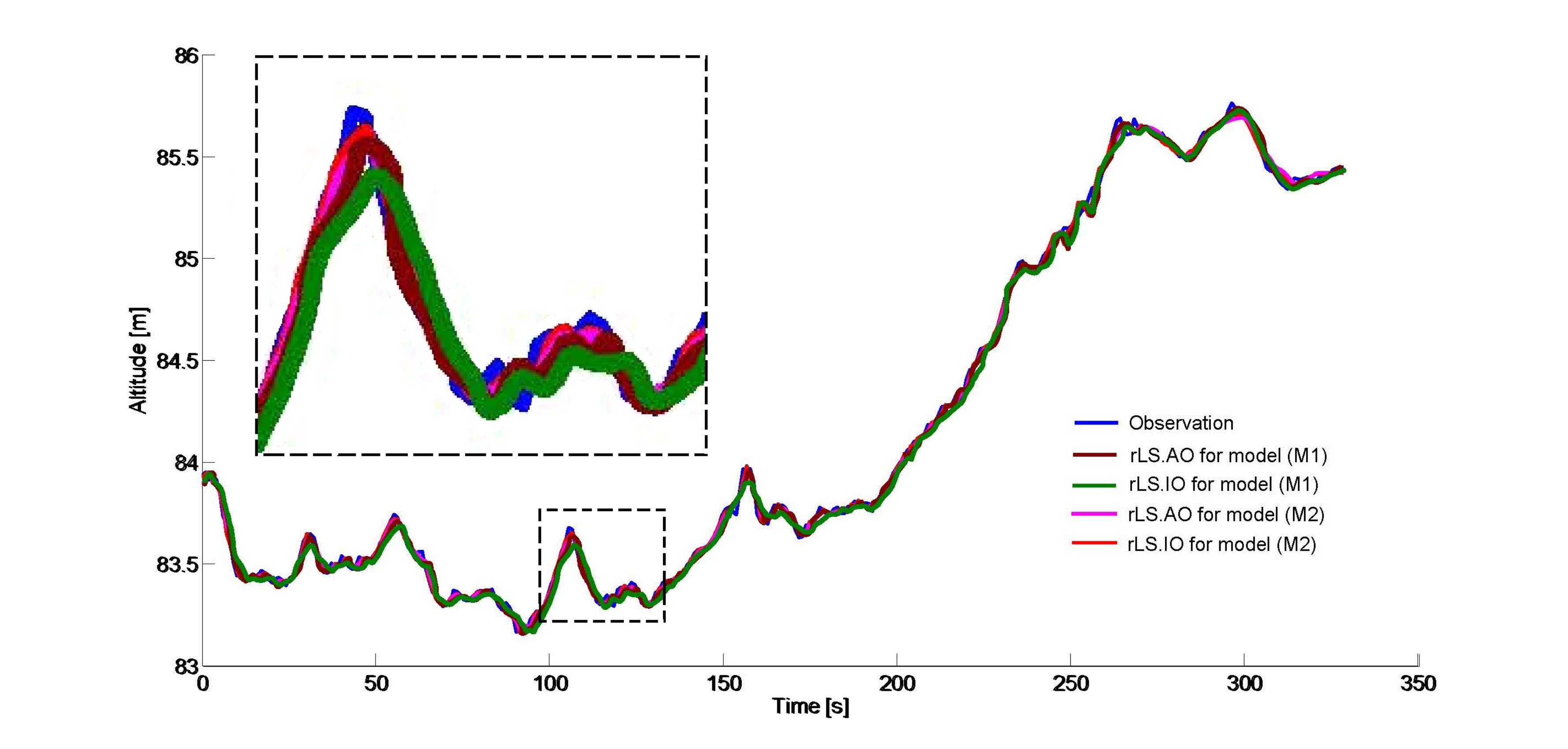}}
\caption{Observations and reconstructed values for different filters at Models~(M1) and (M2)}
\label{plot31}
\end{figure}

Other data sets we analyzed showed much stronger evidence for outliers,
but for confidentiality reasons, we limit ourselves to this data set and
instead discuss the behavior of our procedures at outliers in generated data.
We may mention though, that our procedures work well in the discussed
problems.

  \item[(M2)] {\bf A linear time-varying model} with the following state hyper-parameters:
\begin{equation*}
X_{t}=\left(
  \begin{array}{ccc}
    h_{t} \\
    \alpha_{t} \\
    \dot{\alpha}_{t} \\
     \end{array}
\right) \quad
F_{t}=\left(
  \begin{array}{ccc}
    1 & sp_{t}\Delta t & 0 \\
    0 & 1 & \Delta t\\
    0 & 0 & 0 \\
     \end{array}
\right) \quad
v_{t}\sim N_{3}(0,Q),
\quad
Q=\textrm{diag}(0,0,0.05).
\end{equation*}
Note that to complete the derivation of the model, we need to define a
state matrix for $t=0$, and since we do not have the vehicle
speed at time $t=0$ we set it to the first observation
$sp_{1}$, i.e.
\begin{equation*}
F_{0}=\left(
  \begin{array}{ccc}
    1 & sp_{1}\Delta t & 0 \\
    0 & 1 & \Delta t\\
    0 & 0 & 0 \\
     \end{array}
\right).
\end{equation*}

The hyper-parameters of the observation equation in this model are the following:
\begin{equation*}
Y_{k}=\left(
  \begin{array}{cc}
    \widetilde{h}_{k} \\
    \widetilde{\dot{\alpha}}_{k} \\
     \end{array}
\right) \quad
Z=\left(
  \begin{array}{ccc}
    1 & 0 & 0 \\
    0 & 0 & 1 \\
     \end{array}
\right) \quad
\varepsilon_{t}\sim N_{2}(0,V),
\quad V=\textrm{diag}(5,0.005).
\end{equation*}

The initial distribution of the state vector is  again defined as follows
\begin{equation*}
a_{0}=(h_{1}, 0, 0)', \quad Q_{0}=\textrm{diag}(5,0.005,0.005).
\end{equation*}

  \item[(M3)] {\bf A quadratic time-invariant model} accounting for acceleration.
  This gives a nonlinear SSM in the notation of Section~\ref{EKF},
  specified as
\begin{equation*}
f_t(X_{t-1},u_t,v_t)=A(X_{t-1}\otimes X_{t-1})+BX_{t-1}+v_{t},
\end{equation*}
\begin{equation*}
z_t(X_t,w_t,\varepsilon_t)=ZX_{t}+\varepsilon_{t}
\end{equation*}
Here $\otimes$ denotes the Kronecker product, i.e.,
the quadratic term can be written in the following form
\begin{equation*}
A(X_{t}\otimes X_{t})=\sum_{l=1}^{p}A^{l}[X_{t}]_{l}X_{t},
\quad k=0,..,T-1
\end{equation*}
with known $p\times p^{2}$-matrix
$A=\left(A^{1}\mid ... \mid A^{p} \right)$.

The hyper-parameters, states and observations of constructed for this application
quadratic time-invariant model are the following
\begin{equation*}
X_{t}=\left(
  \begin{array}{ccccc}
    h_{t} \\
    sp_{t} \\
    \dot{sp}_{t} \\
    \alpha_{t} \\
    \dot{\alpha}_{t} \\
     \end{array}
\right) \quad
A=\left(0_{5\times5}\Big|
  \begin{array}{ccccc}
    0 & 0 & 0 & \Delta t & 0 \\
    0 & 0 & 0 & 0 & 0\\
    0 & 0 & 0 & 0 & 0\\
    0 & 0 & 0 & 0 & 0\\
    0 & 0 & 0 & 0 & 0\\
    \end{array}
    \Big| 0_{5\times5}\Big| 0_{5\times5}\Big| 0_{5\times5}
    \right),
\end{equation*}

\begin{equation*}
B=\left(
  \begin{array}{ccccc}
    1 & 0 & 0 & 0 & 0\\
    0 & 1 & \Delta t & 0 & 0\\
    0 & 0 & 0 & 0 & 0\\
    0 & 0 & 0 & 1 & \Delta t\\
    0 & 0 & 0 & 0 & 0\\
    \end{array}
    \right) \quad
v_{t}\sim N_{5}(0,Q),
\quad
Q=\textrm{diag}(0,0,2,0,0.005),
\end{equation*}

\begin{equation*}
Y_{t}=\left(
  \begin{array}{cccc}
    \widetilde{h}_{t} \\
    \widetilde{sp}_{t} \\
    \widetilde{\dot{sp}}_{t} \\
    \widetilde{\dot{\alpha}}_{t} \\
     \end{array}
\right) \quad
Z=\left(
  \begin{array}{ccccc}
    1 & 0 & 0 & 0 & 0\\
    0 & 1 & 0 & 0 & 0\\
    0 & 0 & 1 & 0 & 0\\
    0 & 0 & 0 & 0 & 1\\
     \end{array}
\right) \quad
\varepsilon_{t}\sim N_{4}(0,V),
\quad V=\textrm{diag}(5,2,2,0.005).
\end{equation*}

The initial distribution of the state vector here is defined as follows
\begin{equation*}
a_{0}=(h_{1}, sp_{1},0, 0, 0)', \quad Q_{0}=\textrm{diag}(5,2,2,0.005,0.005).
\end{equation*}
\end{enumerate}

\section{Simulation} \label{SimSect} 
To see how our procedures (and some competitors) work in the outlier setting
they are constructed for, we produce an simulation study done
in {\sf R}, \citet{RMANUAL}, in the framework of our package {\tt robKalman} developed
at \url{http://r-forge.r-project.org/projects/robkalman/}.

For these simulations, we generated 10000 runs of data for two different
models: Model~(SimA), a one-dimensional time-invariant steady-state model
with parameters $F=Z=Q=V=1$, $a_0=1$ and time horizon $T=50$ for comparison with
the median-based filters of Section~\ref{hybridFsf}. More specifically,
we use function {\tt hybrid.filter} from {\sf R} package {\tt robfilter},
with specifications {\tt PRMH} and {\tt MMH} denoted by {\tt hybf} 
and {\tt hybs} below, respectively (for {\it hyb}rid {\it f}ilter / {\it s}moother).
Both procedures are used with fixed window width $5$ and minimum
number of non-missing observations $2$, reflecting the actual outlier
situation.

Second, we study Model~(SimB) which takes over dimensions and typical parameters from
the application of Section~\ref{ApplSect} as specified
in detail in \eqref{Eq:SimSect} and already used in Section~\ref{Sec:NonObsAsp}.

As outlier specification, in both models we use the ones from Section~\ref{devidmod}
with $r_{\rm\SSs IO}=r_{\rm\SSs AO}=0.1$, i.e., \eqref{YSO} and \eqref{indep2} in
the AO case and \eqref{IOSO} in the IO case.
As contaminating distributions we chose
${\rm Cauchy}$ for $X_t^{\rm\SSs di}\sim {\rm Cauchy}(-10,1)$
and $Y_t^{\rm\SSs di}\sim {\rm Cauchy}(5,1)$ in Model~(SimA),
while in Model~(SimB) we used  $X_t^{\rm\SSs di}\sim {\rm multiv.Cauchy}(0,Q)$
and $Y_t^{\rm\SSs di}\sim {\rm Cauchy}(0,1/1000)$ (using {\sf R} packages
{\tt mvtnorm} \citep{mvtnorm} and {\tt MASS} \citep{V:R:02}).

The results for Model~(SimA) are summarized in Table~\ref{TabSimA}
and Figures~\ref{figSimA1}--\ref{figSimA4}, the ones for
Model~(SimB) in Table~\ref{TabSimB}
and Figures~\ref{figSimB1}--\ref{figSimB4}. In the annotation
we denote the smoother versions of rLS.AO and rLS.IO by
SrLS.AO and SrLS.IO, respectively.

In the tables, we display the empirical mean squared error (MSE) for
each of the procedures at each situation. It is clearly
visible that each filter does the job well it is made for:
The classical Kalman filter is best in the ideal situation with rLS.IO
and (a little less so) rLS.AO still close by. In the AO situation,
the AO-robust filter is best (also compared to the median-type filters
which comes second best), whereas classical Kalman, and, even worse rLS.IO, have problems.
In the IO situation the situation changes dramatically: rLS.IO excels,
whereas the classical Kalman filter shows its inertia, and, much worse,
rLS.AO and hybf are not at all able to track these abrupt signal changes.
For smoothing the situation is a bit different: At large we have the same
picture as for filtering, however the IO-robust smoother is not doing
a good job at all: the filter is much better here. This remains to be
studied in more detail. In the ideal and AO situation however the smoothers
do improve the filter (as they should). So it seems a strictly recursive
smoother like the one we propose is well capable to deal with spiky outliers
but much less so to track abrupt changes.

As to the graphics, in Figures~\ref{figSimA1} and \ref{figSimB1},
we display the coordinatewise distribution of the reconstruction errors
$\Delta X_{35}=X_{35|35}-X_{35}$ for filtering and $\Delta X_{35}=X_{35|50}-X_{35}$
for smoothing.  The columns of the panels are the situations (ideal, AO, and IO),
the rows (for Model~(SimB)) the state coordinates. In each panel, we display
the filters first and then the smoothers, ordered as in the columns of the tables.
In order to keep the boxes of the boxplots distinguishable, we skipped all outliers
outside $(-15,15)$ in Model~(SimA) and 8 times the coordinate-wise maximal interquartile
range in Model~(SimB). In addition to the tables of the MSE, these figures reveal
that also in terms of the bulk of the data there are differences among the
procedures:
In Figure~\ref{figSimB1} for Model~(SimB), coordinate~1 is hardest to reconstruct,
and shows large differences between smoother and filter. At least for
coordinates 1 and 2, the central boxes are definitively smallest for the rLS.AO filter
and its smoother in the AO situation, with even a large improvement by the smoother.
The same goes for the rLS.IO filter in the IO situation, whereas as already noted in the tables,
the IO smoother is less convincing in the IO situation, although not really worse
than the filter in terms of the central box, although coordinate 3 tells a different
story here.

In Figure~\ref{figSimA1} for Model~(SimA), we scaled all panels to the same coordinates in order
to be better able to compare the situations; otherwise the conclusions to be drawn
parallel the ones for Model~(SimA), except that in positions 4 an 8 in each panel
we display the median type filters and smoother which lead to largest central boxes
in the filters, and in the smoothers are in between the specialized robust smoother and
the respective unsuitably robustified smoother.

The remaining figures display the distribution of the normed reconstruction
errors, where in spite of Proposition~\ref{PropA6} below, we have not changed the norm
in the IO situation. To visualize all outliers this time, we use a logarithmic
scale for the $y$-axis which also emphasizes ``inliers'' (compare \citet[p.~140]{Ha:Ro..:86}),
i.e., situations where the procedures behave extra-ordinarily well. Per se, for reconstruction, this
is of course a beneficiary situation, but, in case inference is also of interest,
this will lead to underestimation of the true variation.

In Model~(SimA) in the ideal and IO situation, the median-type filters and smoothers
excel in this direction, but also the rLS filters both perform better than
the classical filter in this respect. For the smoother, the classical one
is best, though. To the upper tail, the norms even for the worst situations
surpasses the central box only by little (on log-scale, though). In the AO
situation the specialized robust filters and smoothers have the shortest tail
but as to the boxes are hardly distinguishable from the classical ones, while
in the IO situation only the rLS.IO has a convincing tail.

In Model~(SimB), in the ideal situation, the
smoothers are clearly better than the filters but much less so in the outlier
situations (although the AO smoother is significantly better than the filter).
We also see that the classical procedures produce more inliers than the robustified
ones in the ideal and IO situation, while in the AO situation no clear statement
can be made.

\begin{table}[!ht]
\begin{tabular}{l|rrrr|rrrr}
              situation&\multicolumn{4}{c|}{filter}&\multicolumn{4}{c}{smoother}\\
         &Kalman   &rLS.IO   &rLS.AO  & hybf    &  Kalman   &rLS.IO  &rLS.AO &hybs\\
\hline
id & 0.628 & 0.679&    0.808&   2.180&  0.447&   0.475&    0.575&   0.939\\
AO &24.206 &30.379&    1.446&   4.950& 14.087&  14.226&    1.097&   1.475\\
IO &30.175 & 0.729& 1850.977& 629.129& 66.538&  95.579& 1846.870& 623.371
\end{tabular}
\caption{empirical MSEs in Model~(SimA) at $t=35$ for $T=50$}\label{TabSimA}
\end{table}

\begin{figure}[!ht]
\centerline{\includegraphics[width=0.8\linewidth]{"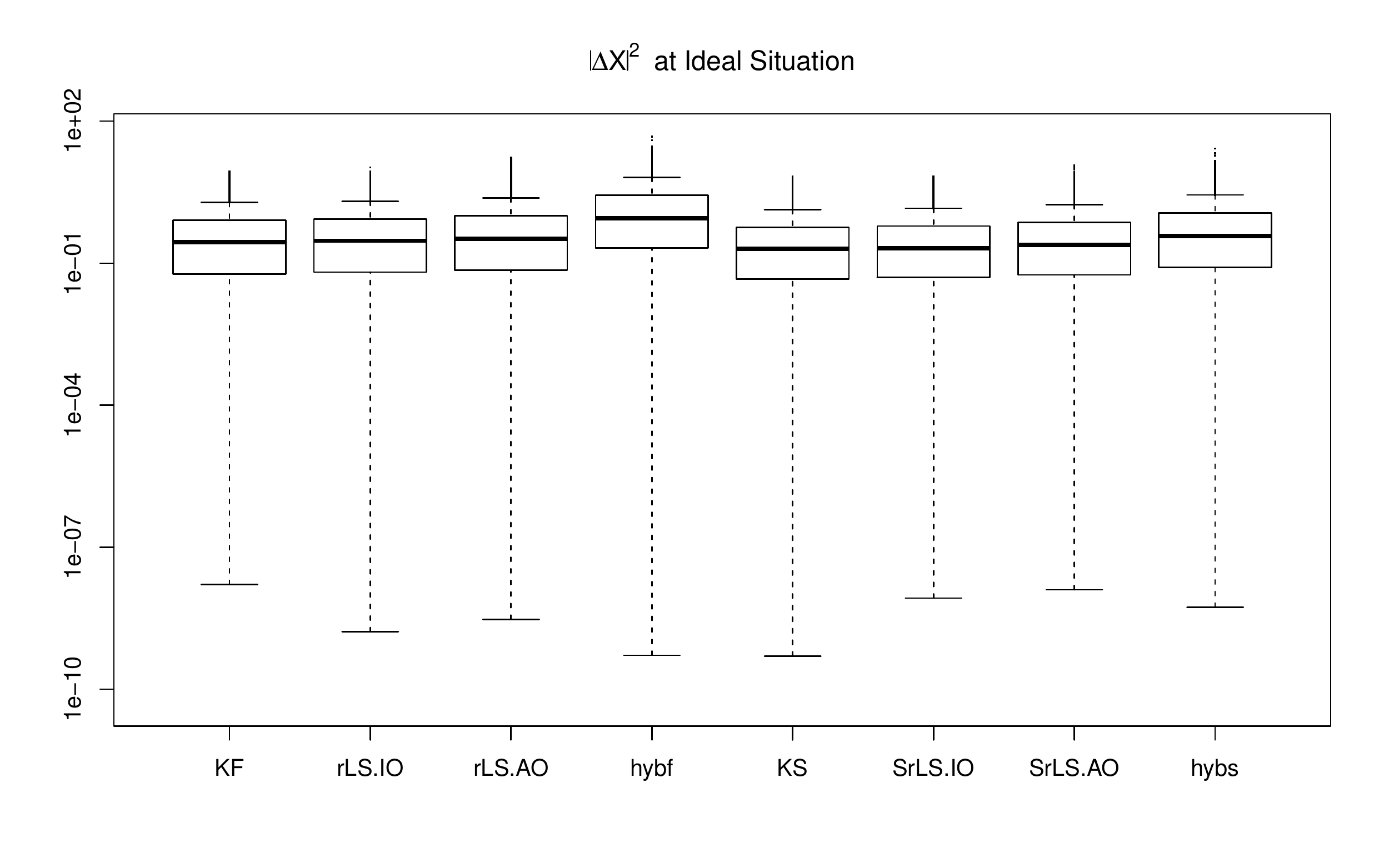}}
\caption{$\|\Delta X_t\|^2$ in Model~(SimA) at $t=35$ for $T=50$ in ideal situation; filtered (boxes 1--4) and smoothed ($T=50$; boxes 5--8) versions}\label{figSimA1}
\end{figure}
\begin{figure}[!ht]
\centerline{\includegraphics[width=0.8\linewidth]{"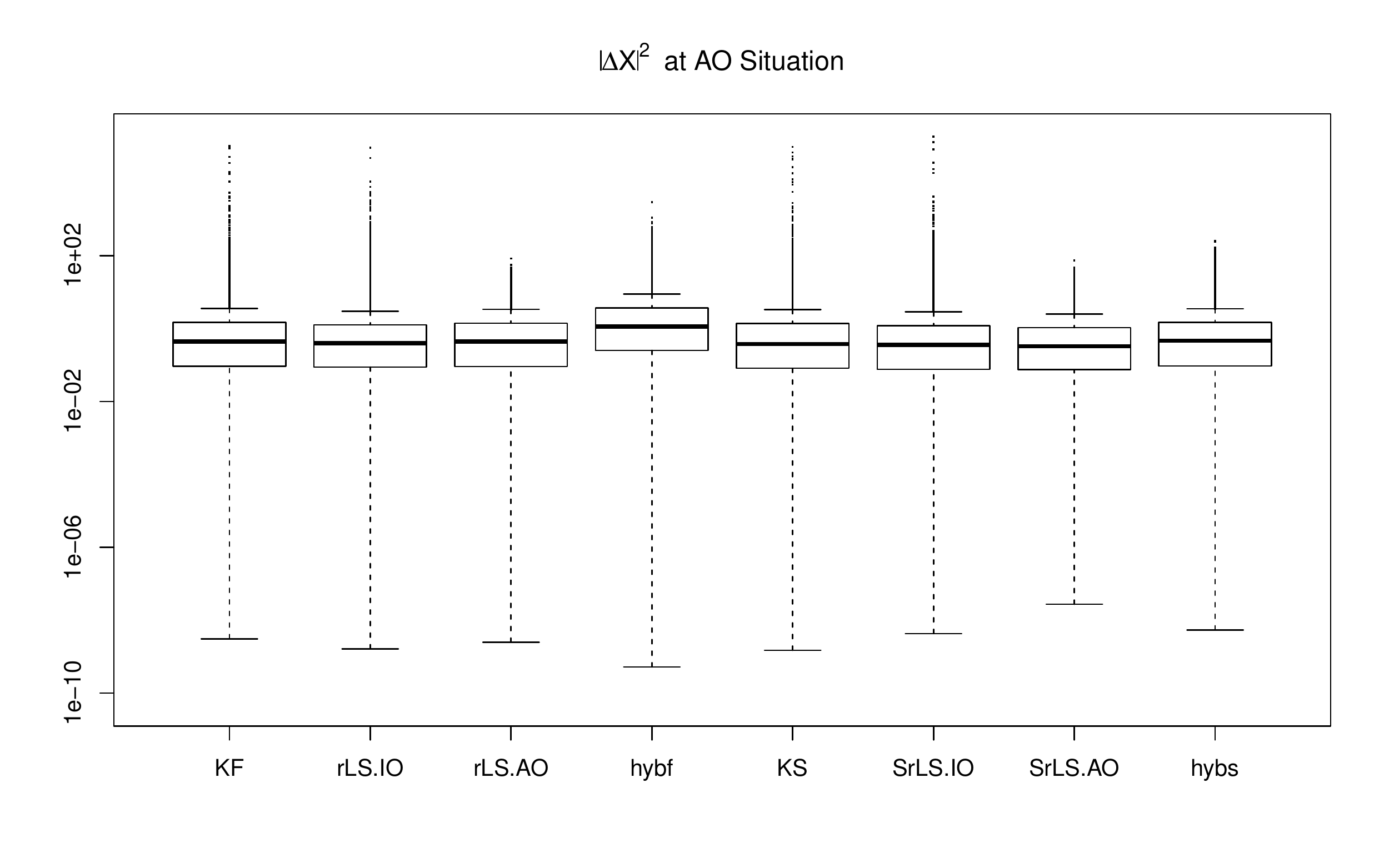}}
\caption{$\|\Delta X_t\|^2$ in Model~(SimA) at $t=35$ for $T=50$ in AO situation; filtered (boxes 1--4) and smoothed ($T=50$; boxes 5--8) versions}\label{figSimA2}
\end{figure}
\begin{figure}[!ht]
\centerline{\includegraphics[width=0.8\linewidth]{"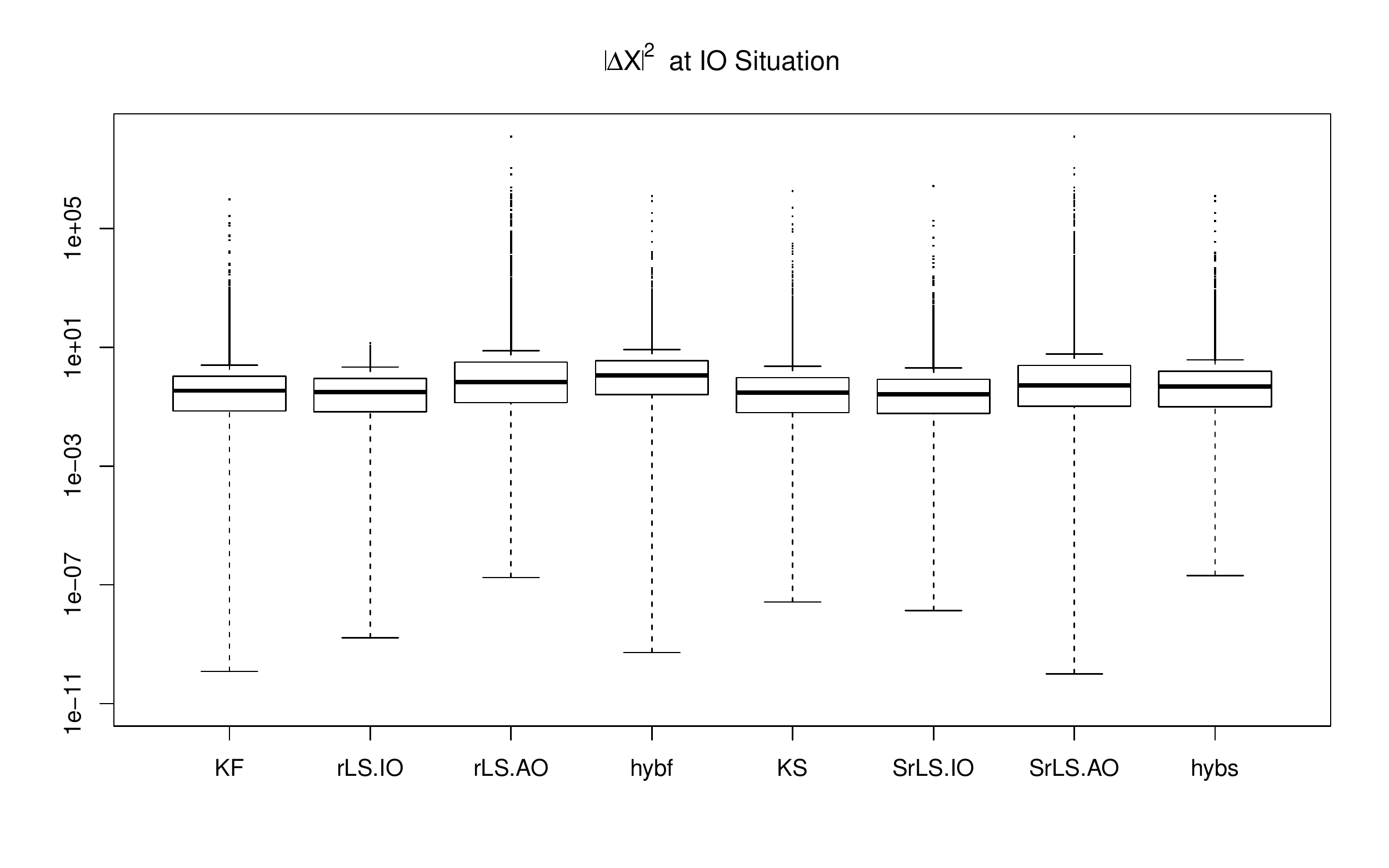}}
\caption{$\|\Delta X_t\|^2$ in Model~(SimA) at $t=35$ for $T=50$ in IO situation; filtered (boxes 1--4) and smoothed ($T=50$; boxes 5--8) versions}\label{figSimA3}
\end{figure}

\begin{figure}[!ht]
\centerline{\includegraphics[width=\linewidth]{"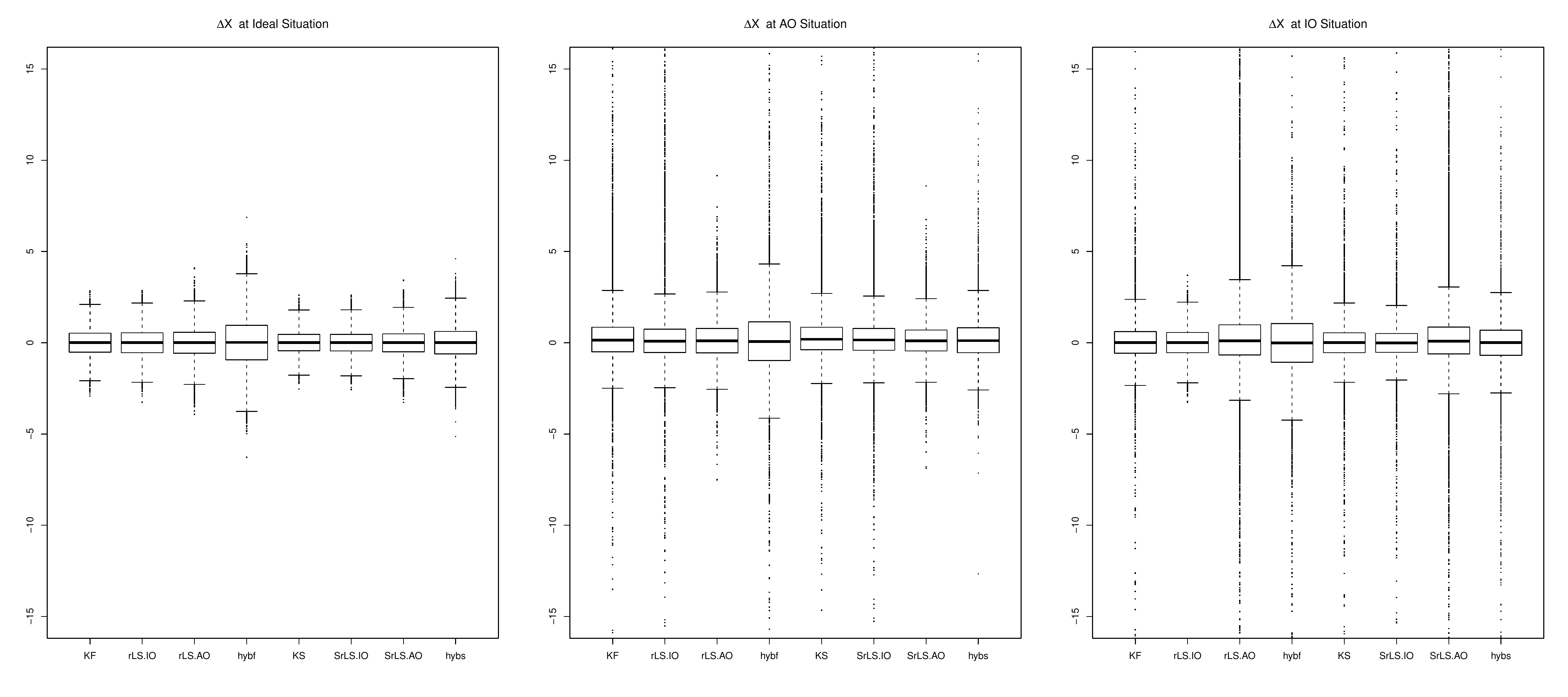}}
\caption{$\Delta X_t$ in Model~(SimA) at $t=35$ in all dimensions and in different situations; filtered (boxes 1--4) and smoothed ($T=50$; boxes 5--8) versions}\label{figSimA4}
\end{figure}

\begin{table}[!ht]
\begin{tabular}{l|rrr|rrr}
              situation&\multicolumn{3}{c|}{filter}&\multicolumn{3}{c}{smoother}\\
         &Kalman   &rLS.IO   &rLS.AO     &  Kalman   &rLS.IO  &rLS.AO\\
\hline
id &   0.034 &   0.047&    1.725&    0.011 &    0.013&    1.514\\
AO &1995.207 &7676.445&    5.661& 5520.957 &32270.199&    5.262\\
IO & 175.484 &   3.553& 7515.495&  116.626 &   86.502& 7513.354
\end{tabular}
\caption{empirical MSEs in Model~(SimB) at $t=35$ for $T=50$}\label{TabSimB}
\end{table}
\begin{figure}[!ht]
\centerline{\includegraphics[width=0.8\linewidth]{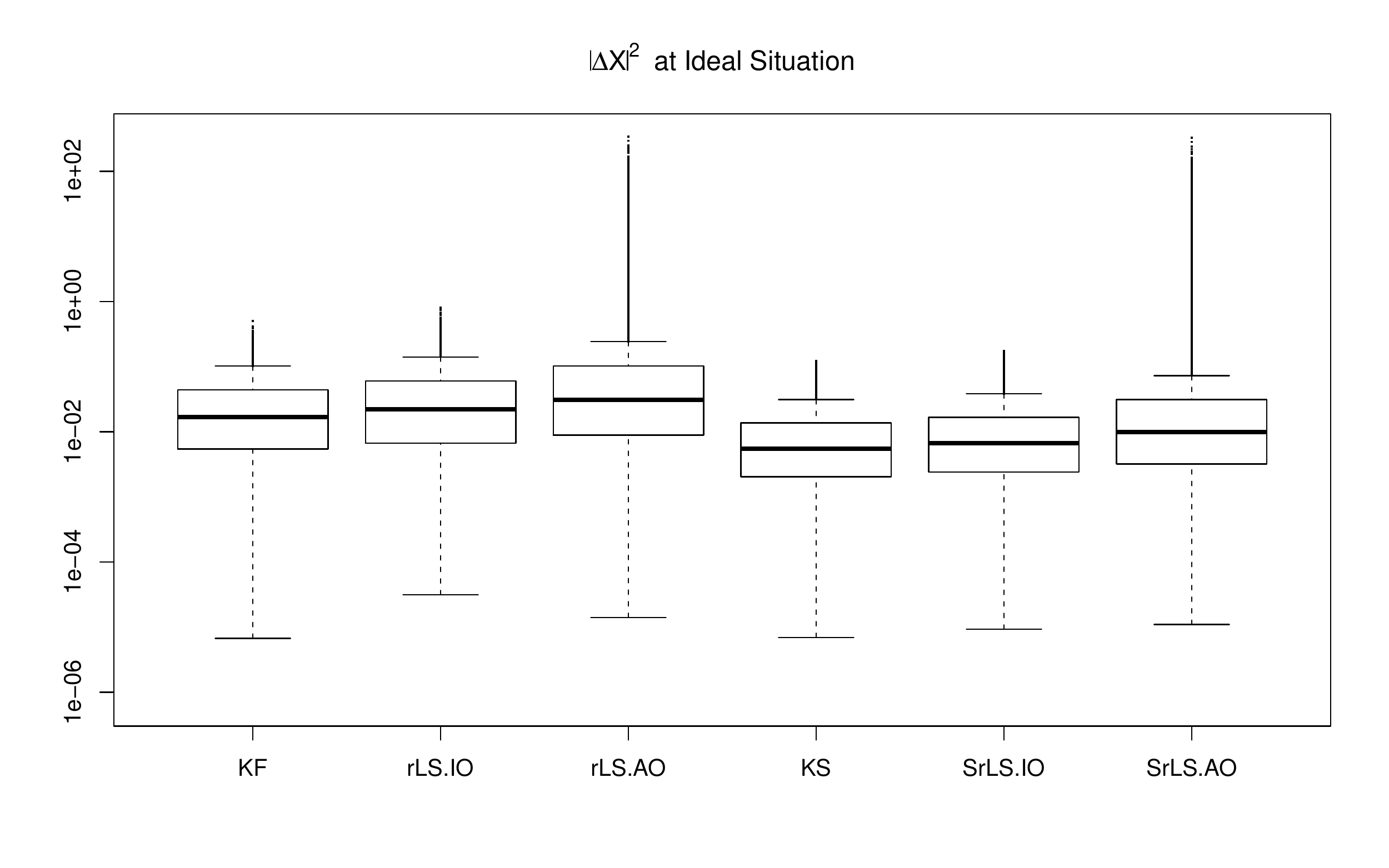}}
\caption{$\|\Delta X_t\|^2$ in Model~(SimB) at $t=35$ for $T=50$ in ideal situation; filtered (boxes 1--3) and smoothed ($T=50$; boxes 4--6) versions}\label{figSimB1}
\end{figure}
\begin{figure}[!ht]
\centerline{\includegraphics[width=0.8\linewidth]{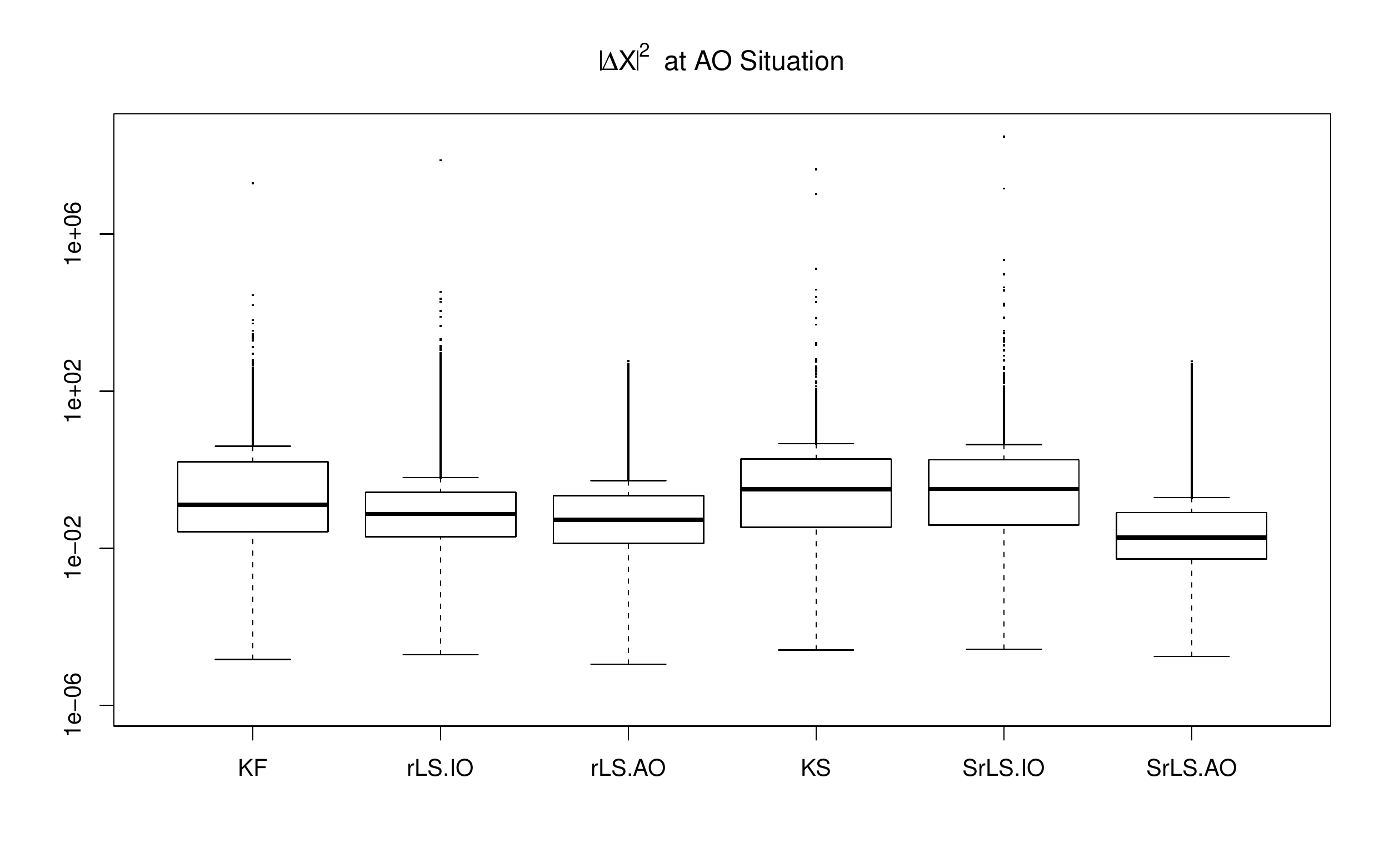}}
\caption{$\|\Delta X_t\|^2$ in Model~(SimB) at $t=35$ for $T=50$ in AO situation; filtered (boxes 1--3) and smoothed ($T=50$; boxes 4--6) versions}\label{figSimB2}
\end{figure}
\begin{figure}[!ht]
\centerline{\includegraphics[width=0.8\linewidth]{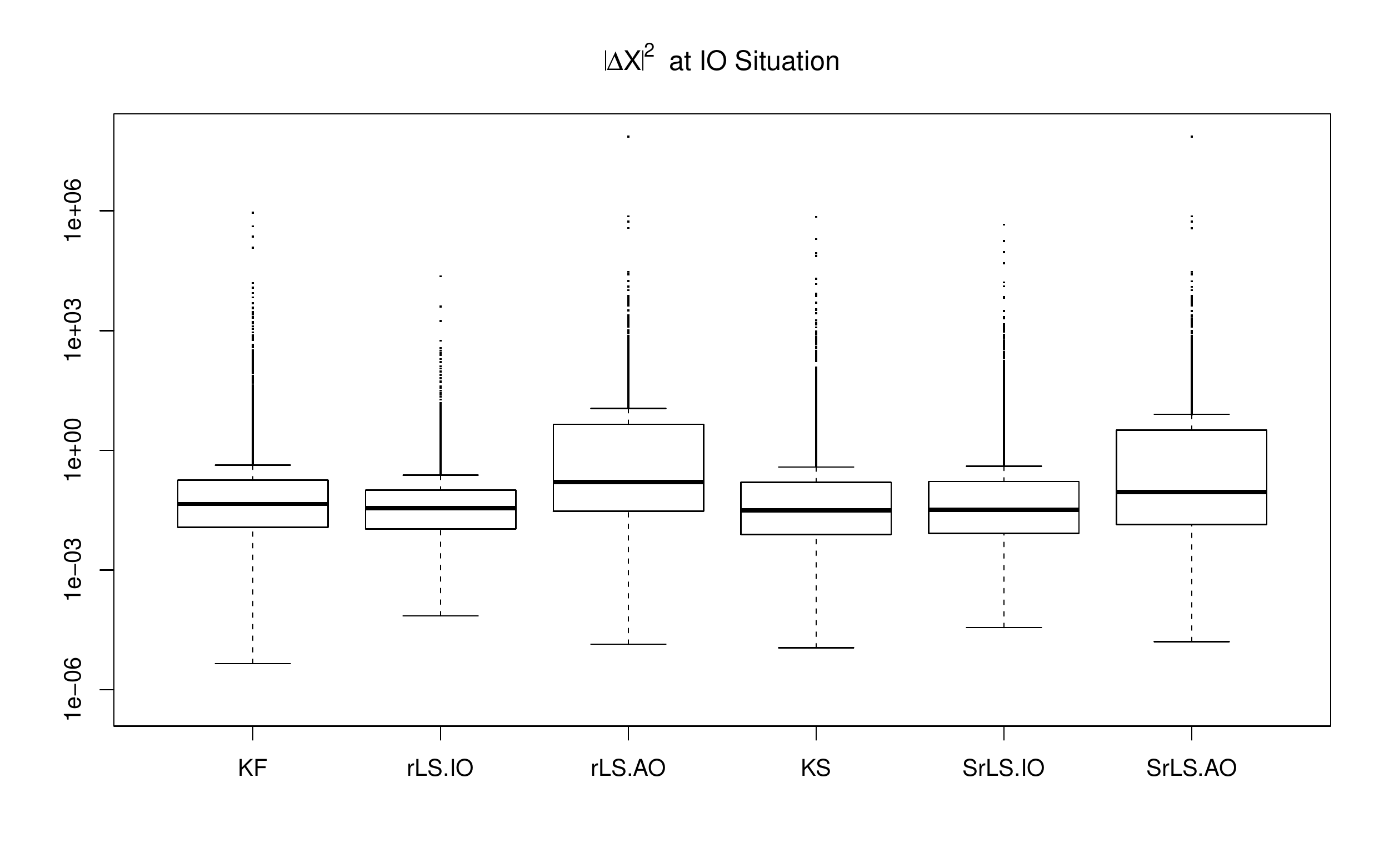}}
\caption{$\|\Delta X_t\|^2$ in Model~(SimB) at $t=35$ for $T=50$ in IO situation; filtered (boxes 1--3) and smoothed ($T=50$; boxes 4--6) versions}\label{figSimB3}
\end{figure}

\begin{figure}[!ht]
\centerline{\includegraphics[width=1.1\linewidth]{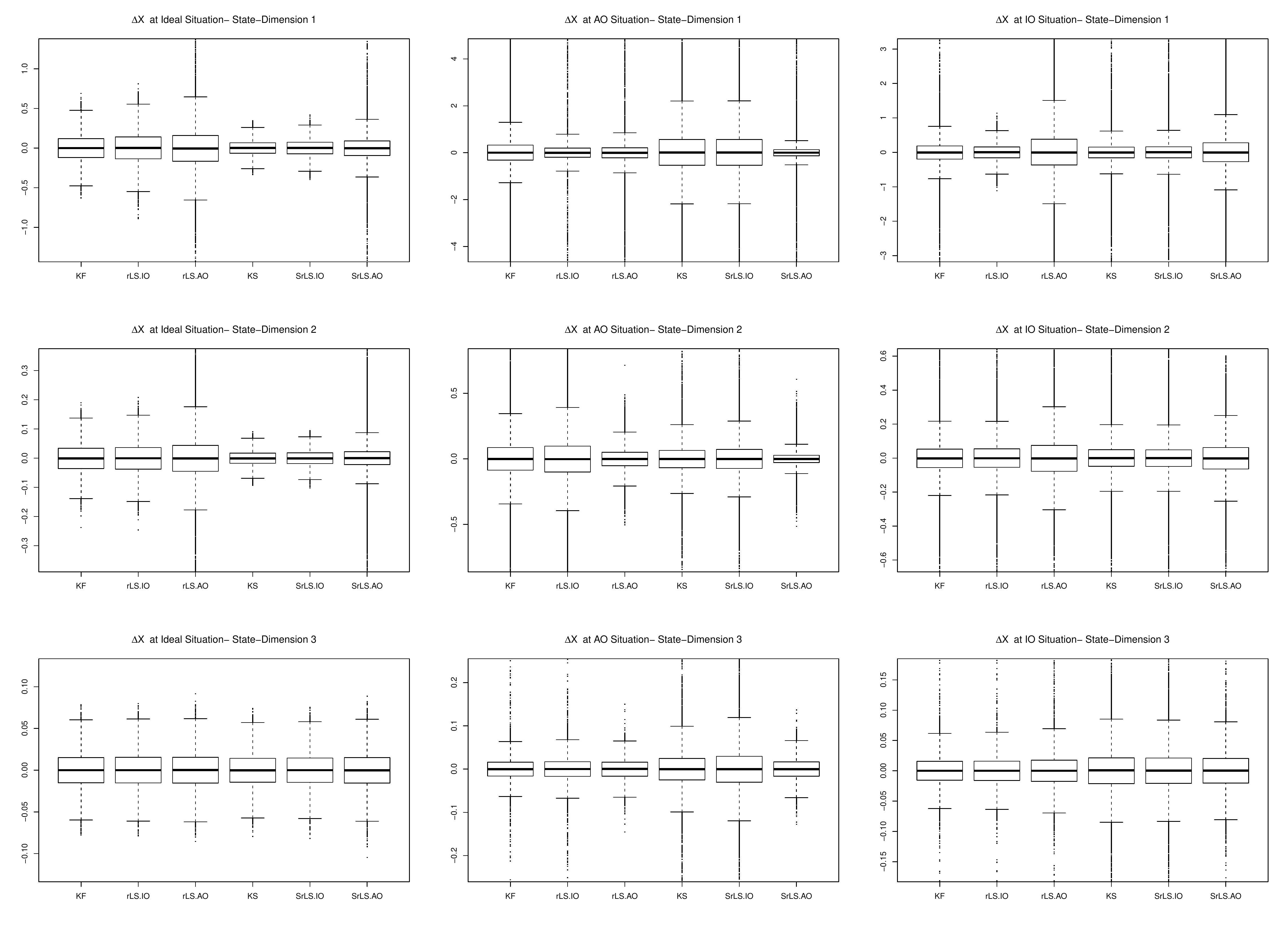}}
\caption{$\Delta X_t$ in Model~(SimB) at $t=35$ in all dimensions and in different situations; filtered (boxes 1--3) and smoothed ($T=50$; boxes 4--6) versions}\label{figSimB4}
\end{figure}

\section{Discussion and conclusion} \label{SumSect}
\paragraph{Contribution}
In this paper we have presented robustifications of the Kalman filter
and smoothers specialized either for damping of spiky outliers or
for faster tracking a deviated series where the smoothers and the
general IO-robust filter are new to the best of our knowledge.

All our procedures are recursive, hence extra-ordinarily
fast, so can be used in online problems and can easily be used
to also robustify the Extended Kalman filter for non-linear
state-space models. Our procedures are implemented to {\sf R},
developed under {\tt r-forge}, \url{http://r-forge.r-project.org/}
and soon will be submitted to {\tt CRAN}, \url{http://cran.r-project.org/}.

We have demonstrated the superior behavior of our procedures at the
outlier situations they are made for and, in a study of stylized
``realistic'' outlier situations we showed that they also can cover
with a wider variety of outlier situations. These stylized situations
also helped to identify certain flaws of the procedures.

In particular the new rLS.AO smoother and the rLS.IO filter seem
to be recommendable already as they are. As we have seen in the
evaluations, further research is needed though to improve
the IO-robust smoother which could not convince so far.

As to the theoretical contributions, for the tracking problem we have
thoroughly settled the case of non invertible observation matrices, i.e.;
the situation when certain directions of the state are not visible.
It turned out that the taken passage to observation space
in order to define the rLS.IO involves no extra costs in terms
of rank defects compared to the classical Kalman filter.
In order to have a well-posed problem though, as demonstrated
in Section~\ref{Stylized}, we need to pass to a semi-norm which
ignores directions, no filter can see.
With this modified criterion, we can establish our IO robust
filter as approximately optimal, where approximately means that
it would be optimal if the ideal conditional expectation were linear.
Now the conditional expectation is not exactly linear, but, as shown
empirically in specific examples, it is close to linear---at least in
a central region.

\paragraph{Outlook}
In reality ``pure'' IO or AO situations hardly occur. Hence we think
it is of high importance to thoroughly study hybrid versions
of our filters (and/or smoothers) combining the
two types of filters (IO and AO) we have discussed so far, to use
them in mixed situations. So far we have only checked a heuristics
based on the sequence of the normed observation residuals $\|\Delta Y_t\|$
in a rolling window. In situation where IOs and AOs are well separated
this already works decently, but the procedure easily gets confused
once in a window we have both IOs and AOs.

In non-linear state space models, the unscented Kalman filter,
compare \citet{W:vM:02} deserves a robustification, which could easily take up
the ideas of this paper.

Finally, robust filters and smoothers
are a key ingredient of the EM algorithm \citep{Sh:St:82} to estimate unknown
(hyper-) parameters with interesting application to the fitting
of stochastic differential equations to financial data.

\section*{Acknowledgements}
The authors thank two anonymous referees for their valuable and helpful comments.
Financial support from VW foundation in the framework of project Robust Risk
Estimation for D.~Pupashenko is gratefully acknowledged.
\small
\appendix
\section{Appendix}
\subsection{Optimality of the classical Kalman filter} \label{OptimClKF}
Optimality of the classical Kalman filter among all linear filters in
$L_2$-sense and, under normality of the error and innovation distributions,
among all measurable filters is a well-known fact, compare, e.g. \citet[Sec.~5.2]{A:M:90}.
As we will need some of the arguments later, let us complement this fact
by some generalization to arbitrary norms generated by a quadratic form
and by a thorough treatment of the case of singularities in the
covariances arising in the definition of the Kalman gain from \eqref{bet3}.
To do so, we take the orthogonal decomposition of the Hilbert into closed linear
subspaces as
${\rm lin}(Y_{1:(t-1)})\oplus {\rm lin}(\Delta Y_{t})$ as granted
and  for $X=X_t-X_{t|t-1}$ and $Y=\Delta Y_t$ as given in \eqref{bet3},
derive $\hat K_t$.
To this end, for any matrix $A$ let us denote by $A^-$ the
generalized inverse of $A$ with the defining properties
\begin{equation}
A^-AA^-=A^-,\quad AA^-A = A,\qquad A^-A = (A^-A)^\tau,\quad AA^- = (AA^-)^\tau
\end{equation}
and for $D$ a positive semi-definite symmetric matrix in $\R^{p\times p}$ and
for $x\in\R^p$ define the semi-norm generated by $D$ as $\|x\|^2_D:=x^\tau D^{-} x$.

\begin{Lem} \label{lem1}
Let $p,q \in \N$, $P$ some probability and $X\in L_2^p(P)$, $Y\in L^q_2(P)$,
$\Ew X=0$, $\Ew Y=0$, where for some $Z\in\R^{q\times p}$, and some
$\ve \in L_2^q(P)$ independent of $X$, $Y=ZX+\ve$. Let $D$ a positive semi-definite
symmetric matrix in $\R^{p\times p}$.
Then
\begin{equation}
\hat K = \Cov(X,Y) \Cov(Y)^-
\end{equation}
solves
\begin{equation} \label{optL2}
\Ew \|X-KY\|^2_D = \, \min{}!,\qquad K \in \R^{p\times q}
\end{equation}
$\hat K$ is unique up to addition of some $A \in \R^{p\times q}$ such that $A\Cov Y =0$
and some $B \in \R^{p\times q}$ such that $DB=0$.
If  $\hat K = D D^- \hat K$, $\hat K$ has smallest Frobenius norm among all solutions $K$ to
\eqref{optL2}.
\end{Lem}

\begin{proof}{}
Denote $L_2^p(P,D)$ the Hilbert space generated by
all $\R^p$ valued random variables $U$ such that
$\Ew_P \|U\|_D^2 <\infty$---after a passage to equivalence
classes of all random variables $U, U'$ such that
$\Ew_P \|U-U'\|_D^2 = 0$.
Let $S=\Cov(X)$ and $V=\Cov(\ve)$. Then $\Cov(X,Y)=SZ^\tau$ and $\Cov(Y)=ZSZ^\tau +V$.
Denote the approximation space $\{ KY\,\mid\, K\in \R^{p\times q}\}\subset L_2^p(P,D)$ by
${\cal K}$. ${\cal K}$ is a closed linear subspace of $L_2^p(P,D)$, hence by
\citet[Thm.~4.10]{Ru:87} there exists a unique minimizer $\hat X = \hat K Y \in {\cal K}$ to
problem~\eqref{optL2}. It is characterized by
\begin{equation} \label{normEq}
\Ew (X- \hat KY)^\tau D^{-} KY = 0,\qquad \forall K \in \R^{p\times q}
\end{equation}
Plugging in $K=D e_i \tilde e_j^\tau$, $\{e_i\}$,  $\{\tilde e_j\}$ canonical bases
of $\R^p$, $\R^q$, respectively, we see that \eqref{normEq} is equivalent to
\begin{equation} \label{normEq2}
\Ew \pi (X- \hat KY)Y^\tau = 0 \quad\iff\quad \pi K \Cov(Y)= \pi \Cov(X,Y)
\end{equation}
where $\pi=D^-D$ is the orthogonal projector onto the column space of $D$.
But $y\in \R^q$ can only lie in $\ker \Cov(Y)$ if $y\in \ker \Cov(X,Y)$.
Hence indeed $\hat K \Cov(Y)=\Cov(X,Y)$, and the uniqueness assertion is obvious.
We write  $\pi_D$ and $\pi_{C}$ for the orthogonal projectors to the column spaces
of $D$ and $\Cov(Y)$, respectively and $\bar\pi_{C}=\EM_q-\pi_{C}$, $\bar \pi_D=\EM_p-\pi_D$
for the corresponding complementary projectors.
Then we see that $\hat K = \hat K \pi_{C}$ and for any $A \in \R^{p\times q}$ with
$A\Cov Y =0$ we have $A=A\bar \pi_{C}$ and for any $B \in \R^{p\times q}$ with
$DB =0$ we have $B=\bar \pi_{D} B$; hence
\begin{eqnarray*}
\|\hat K + A + B\|^2 &=& \tr \hat K^\tau \hat K + 2\tr A ^\tau \hat K + 2\tr B^\tau \hat K +
\tr (A+B)^\tau (A+B)\\
&=& \|\hat K \|^2 + 2\tr  \hat K \pi_{C} \bar \pi_{C} A ^\tau + 2\tr  \hat K \pi_{D} \bar \pi_{D} B ^\tau+ \| A+B \|^2  \\
&=& \|\hat K \|^2 + \| A +B \|^2  \geq \|\hat K \|^2\qquad \mbox{with equality iff $A+B=0$.}
\end{eqnarray*}
\end{proof}

\subsection{Sketch of the optimality of the rLS.AO} \label{rLS.AO.opt}
\paragraph{(One-Step)-optimality of the rLS}
The rLS filter is  optimally-robust in some sense: To see this,
in a first step we essentially boil down our SSM to \eqref{simpAdd}, i.e., we
have an unobservable but interesting state $X\sim P^X(dx)$,
where for technical reasons we assume that in the ideal
model $\Ew |X|^2 <\infty$.
Instead of $X$, for some $Z\in\R^{q\times p}$, we rather observe the sum $Y=ZX+\ve$
of $X$ and a stochastically independent error  $\ve$.
As (wide-sense) AO model,  we consider the SO outlier
of \eqref{YSO}, \eqref{indep2}. The corresponding neighborhood
is defined as
\begin{equation}\label{U-SO}
{\cal U}^{\rm\SSs SO}(r)=\bigcup_{0\leq s\leq r}
\Big\{{\cal L}(X,Y^{\rm\SSs re}) \,|\, Y^{\rm\SSs re} \;
\mbox{acc. to \eqref{YSO} and \eqref{indep2} with radius $s$}\Big\}
\end{equation}
In this setting we may formulate two typical robust optimization problems,
i.e., a minimax formulation, and, in the spirit of \citet[Lemma~5]{Ha:68},
a formulation where robustness enters as side condition on the bias to
be fulfilled on the whole neighborhood
\begin{align}
\mbox{[Minmax-SO]}\quad&\max\nolimits_{{\cal U}}\, \Ew_{\SSs\rm re} |X-f(Y^{\rm\SSs re})|^2 = %
      \min\nolimits_f{}! \label{minmaxSO}\\
\mbox{[Lemma-5]}\quad& \Ew_{\SSs\rm id} |X-f(Y^{\rm\SSs id})|^2 = \min\nolimits_f{}! \quad
   \mbox{s.t.}\;\sup\nolimits_{\cal U}\big|\Ew_{\SSs\rm re} f(Y^{\rm\SSs re})-
                \Ew X \big|\leq b \label{Lem5SO}
\end{align}
Then one can show that setting $D(y)=\Ew_{\SSs\rm id}[X|Y=y]-\Ew X$,
the solution to both problems is $\hat f(y)=\Ew X +H_\rho(D(y))$
(with $b=\rho/r$ in Problem~\eqref{Lem5SO}), and that this
is just the  (one-step) rLS, once $\Ew_{\SSs\rm id}[X|Y]$ is linear in $Y$.
A proof to this assertion is given in \citet[Thm.~3.2]{Ru:10b}.

\begin{Rem}\label{rem33}
  \begin{ABC}
\item
As mentioned in Section~\ref{rLSsec.AO}, \citet{C:H:11} show an
optimality similar to the one for Problem~\eqref{Lem5SO}, and hence,
non-surprisingly come up with a similar procedure.
\item
The  ACM filter by \citet{M:M:77}, an early competitor
to the rLS, by analogy applies \cite{Hu:64}'s minimax variance result
to the ``random location parameter $X$'' setting of \eqref{simpAdd}.
They come up with redescenders as filter $f$. Hence the ACM filter is
not so much vulnerable in the extreme tails but  rather where the
corresponding $\psi$ function takes its maximum in absolute value.
Care has to be taken, as such ``inliers'' producing the least
favorable situation for the ACM are much harder to detect on
na\"ive data inspection, in particular in higher dimensions.
\item
For exact SO-optimality of the rLS-filter, linearity of
the ideal conditional expectation is crucial. However, one can
show that
$\Ew_{\rm\SSs id}[\Delta  X|\Delta Y]$ is linear iff $\Delta X$
is normal, but, having used the rLS-filter in the $\Delta X$-past,
normality cannot hold, see \citet[Prop.'s~3.4,~3.6]{Ru:10b}.
\item
Although rLS fails to be SO-optimal for $t>1$, it does performs quite well
at both simulations and real data. To some extent this can be explained by
passing to a certain extension of the original SO-neighborhoods.
For details see \cite[Thm.~3.10, Prop.~3.11]{Ru:10b}.
 \end{ABC}
\end{Rem}

\subsection{Optimality of the rLS.IO}\label{OptRLSIO}
This section discusses (one-step) optimality of the rLS.IO in some detail.
We omit time indices and write $\Sigma$ for $\Sigma_{t|t-1}$.
To start, let us again look at the boiled
down model \eqref{simpAdd} where we interchange the
r\^ole of $\ve$ and $X$, and note that $X-f(Y)=\ve-g(Y)$ for
$f(Y)=Y-g(Y)$. Hence in this simple model, the optimal reconstruction
of a corrupted $X$ assuming that $\ve$ is still from the ideal
distribution is just $Y-g(Y)$, $g(Y)$ the optimal reconstruction
of $\ve$ in the same situation.

In notation, let us write $\mathop{\rm oP}(a|b)$ for the best linear
reconstruction of $a$ by means of $b$, i.e., the orthogonal
projection of $a$ onto the closed linear space generated by $b$.

Assuming linear conditional expectations and mutatis mutandis
in \citet[Thm.~3.2]{Ru:10b}, the optimally-robust reconstruction
of $\ve$ given $Y$ in the sense of Problems~\eqref{minmaxSO}, \eqref{Lem5SO} is
just $H_b(\mathop{\rm oP}(\ve|Y))$---with the same caveats as
to the optimality for larger time indices as in Remark~\ref{rem33}. But again,
$\mathop{\rm oP}(\ve|Y)=\mathop{\rm oP}(Y-X|Y)=Y-\mathop{\rm oP}(X|Y)$
so the IO-optimal procedure $f_{\rm \SSs IO}$ is
\begin{equation}
f_{\rm \SSs IO}(Y)=Y-H_b(Y-\mathop{\rm oP}(X|Y))
\end{equation}

Details as to the translation of the contamination neighborhoods
and exact formulations of the optimality results are given in \citet{Ru:10b}.

The general setup with some arbitrary $Z\in\R^{q\times p}$,
where $Z$ in general is not invertible, and moreover, even  $Z^\tau \Sigma Z$ may be
singular, is not trivial, though. For instance, our preceding argument so far only covers
reconstruction of $ZX$, but at this stage it is not obvious how to optimally derive
a reconstruction of $X$ from this. In particular, in this general case, there are directions
which our (robustified) reconstruction cannot see---at least all directions
in $\ker Z$. So an unbounded criterion like MSE would play havoc once
unbounded contamination happens in these directions.
So in this context, the best we can do is optimally reconstructing $ZX$ on
the whole neighborhood generated by outliers in $X$ and then, in a second step,
for this best reconstruction of $ZX$, find the best back-transform to
$X$ in the ideal model setting. The question is how much we loose by this.
To this end, note that
\begin{equation}
\mathop{\rm oP}(\ve|Y)=\mathop{\rm oP}(Y-ZX|Y)=Y-Z\mathop{\rm oP}(X|Y)=(\EM_q-ZK)Y
\end{equation}

For $Z^\Sigma$ from \eqref{ZSigma}, we introduce the orthogonal projector
onto the column space of $Z^\tau \Sigma Z$ and its orthogonal
complement as
\begin{equation}
\pi_{Z,\Sigma} = Z Z^\Sigma,\qquad \bar \pi_{Z,\Sigma}:= \EM_q-\pi_{Z,\Sigma}
\end{equation}
Then we have the following Lemma:

\begin{Lem} \label{Lem2}
\begin{ABC}
\item For any positive definite $D$, $Z^\Sigma$ from \eqref{ZSigma} solves
\begin{equation}
\Ew_{\rm \SSs id} \|X- A \mathop{\rm oP}(Z X|Y)\|_D^2=\;\min{}!,\qquad A\in \R^{p\times q}
\end{equation}
\item $\Sigma Z^\tau \bar\pi_{Z,\Sigma} = 0$;
in particular, no matter of the rank
of $Z$ or $\pi_{Z,\Sigma}$, with $K=\Sigma Z^\tau C^-$,
\begin{equation} \label{nothinglost}
Z^{\Sigma} Z K = K
\end{equation}
\end{ABC}
\end{Lem}

\begin{proof}{}
\begin{ABC}
\item As in Lemma~\ref{lem1}, we see that
\begin{equation}
\hat A=\Sigma Z^\tau K^\tau Z^\tau (ZK\Cov(Y)K^\tau Z^\tau)^-
\end{equation}
Abbreviating $Z \Sigma Z^\tau$ by $B$ and $\Cov(Y)$ by $C$, this gives
$
\hat A=\Sigma Z^\tau C^- B(BC^-B)^-
$, and with $\Sigma_{.5}$ the symmetric root of $\Sigma$, and with $G=\Sigma_{.5}Z^\tau$,
this becomes $\hat A=\Sigma_{.5} G C^- G^\tau G(G^\tau G C^-G^\tau G)^-$. Next we pass to the singular
value decomposition of $G=USW^\tau$, with $U$, $W$ corresponding orthogonal matrices in $\R^{p\times p}$ and
$\R^{q\times q}$, respectively, and $S\in \R^{p\times q}$ a matrix with the singular values
on the ``diagonal entries'' $S_{i,i}$, $i=1,\ldots,\min(p,q)$ and $S_{i,j}=0$, $i\not=j$; furthermore, $S_{i,i}>0$
for $i=1,\ldots,d$, $d\leq \min(p,q)$ and $0$ else.  Using $(aba^\tau)^-=(a^\tau)^{-1}b^- a^{-1}$ for $a$ invertible
and setting $T=S^\tau S $, we obtain
$$\hat A=\Sigma_{.5} USW^\tau C^- W T W^\tau(WTW^\tau C^-WTW^\tau)^-=
\Sigma_{.5} USW^\tau C^- W T (TW^\tau C^-WT)^-W^\tau
$$
As the expressions of the symmetric matrices $W^\tau C^- W$ are surrounded by $S$ (resp.\ $T$-)-terms,
we may replace them with a matrix $R\in\R^{q\times q}$ with only entries in the upper $d\times d$ block, i.e.,
$\hat A=\Sigma_{.5} US R T (T R T)^- W^\tau$
and as $R$ now is compatible with $S$ and $T$, $$\hat A= \Sigma_{.5} US R 1_d R^- T^- W^\tau =
\Sigma_{.5} US R R^- T^- W^\tau,\qquad \mbox{for}\quad 1_d=TT^-$$ Now, as $C=WTW^\tau + V$, $W^\tau C^- W=(T+ W^\tau V W)^-$,
in particular the upper $d\times d$ block $R_d$ of $R=1_d (T+ W^\tau V W)^- 1_d$ is invertible and
$$\hat A=\Sigma_{.5} US T^- W^\tau (=\Sigma_{.5} US^- W^\tau)= \Sigma_{.5} USW^\tau W T^-W^\tau = \Sigma Z^\tau B^- = Z^\Sigma$$
\item We start by noting that $\Sigma Z^\tau \bar\pi_{Z,\Sigma} = \Sigma_{.5} USW^\tau W (\EM_q-1_d) W^\tau =0 $.
For \eqref{nothinglost},  we write
$
K=\Sigma Z^\tau (\pi_{Z,\Sigma}+\bar \pi_{Z,\Sigma}) C^- = \Sigma Z^\tau \pi_{Z,\Sigma} C^- = Z^\Sigma Z K
$.
\end{ABC}
\end{proof}

As a consequence of assertion (b) in the preceding Lemma, we obtain
\begin{Cor} \label{CorIO}
No matter of the rank of $Z$ or $\pi_{Z,\Sigma}$,
\begin{equation}
\mathop{\rm oP}(X|Y) = Z^\Sigma(Y -\mathop{\rm oP}(\ve |Y) )
\end{equation}
that is, we can exactly recover  $\mathop{\rm oP}(X|Y)$ from $\mathop{\rm oP}(\ve |Y)$,
and passing over the reconstruction of $ZX$ first
does not cost us anything in efficiency compared to the direct route.
\end{Cor}
\begin{proof}{}
We only note that $Y -\mathop{\rm oP}(\ve |Y)=\mathop{\rm oP}(ZX |Y) = Z K Y$.
\end{proof}

To keep things well-defined in this setting where we have ``invisible directions''
in the state, we may recur to passing to a semi-norm in $X$-space which ignores
such directions. A possible candidate for $D$ in Lemma~\ref{lem1} is
\begin{equation} \label{Dnorm}
D^-=(Z^\Sigma Z)^\tau \Sigma^- Z^\Sigma Z
\end{equation}
On the one hand, as we show below, invisible directions get ignored,
on the other hand, by \eqref{nothinglost}, no direction visible for
the classically optimal procedure is lost.

\begin{Prop} \label{PropA6}
Using $D$ from \eqref{Dnorm} and assuming observation errors
from the ideal situation, maximal
MSE error for rLS.IO measured in this norm remains bounded for IO contamination.
With this norm, $\hat K$ is smallest possible solution
to \eqref{optL2} in Frobenius norm.
\end{Prop}
\begin{proof}{}
The error term $e=X-\hat X$ for the rLS.IO can be written as
$$e=X-Z^\Sigma(Y-H_b(Y-ZKY))=
(\EM_p-Z^\Sigma Z)X -Z^\Sigma (\ve - H_b((\EM_q-ZK)Y)) $$
As $(Z^\Sigma Z)^2 = Z^\Sigma Z$,  we see that
$(\EM_p-Z^\Sigma Z)(Z^\Sigma Z)=0$,
so that in $D$-semi-norm, the $(\EM_p-Z^\Sigma Z)X$ terms cancel out
and we get
\begin{eqnarray*}
e^\tau D^- e &=&[Z^\Sigma (\ve - H_b(\,\cdot\,))]^\tau
(Z^\Sigma Z)^\tau \Sigma^- Z^\Sigma Z [Z^\Sigma (\ve - H_b(\,\cdot\,))] =\\
&=&  (\ve - H_b(\,\cdot\,))^\tau (Z^\Sigma)^\tau(Z^\Sigma Z)^\tau \Sigma^-
Z^\Sigma Z Z^\Sigma (\ve - H_b(\,\cdot\,)) =\\
&=& (\ve - H_b(\,\cdot\,))^\tau B^- (\ve - H_b(\,\cdot\,)) \le
2 \ve^\tau B^- \ve + 2 H_b(\,\cdot\,)^\tau B^- H_b(\,\cdot\,)
\end{eqnarray*}
so MSE is bounded by $2 \tr (B^- (V + b^2 \EM_q))$. The second assertion
is an immediate consequence of Lemma~\ref{lem1} and Lemma~\ref{Lem2}(b).
\end{proof}
Note that changing the norm in the $Y$-space is not necessary for boundedness
reasons, as with only ideally distributed $\ve$, the reconstruction of
$ZX$ can be achieved such that no matter how largely $\Delta X$ is contaminated,
the maximal MSE remains bounded.

\end{document}